\documentclass[11pt,onecolumn]{IEEEtran}

\usepackage{setspace}
\usepackage[pdftex]{graphicx}
\DeclareGraphicsExtensions{.pdf}

\doublespace
\usepackage[]{amsmath}
\usepackage{amssymb}
\usepackage{algorithm}
\usepackage{algorithmic}
\usepackage{bm}
\usepackage{url}
\usepackage{cite}

\usepackage{bm}
\usepackage{color}

\makeatletter \newcommand \listoftodos{\section*{Todo list} \@starttoc{tdo}}
\newcommand\l@todo[2]
  {\par\noindent \textit{#2}, \parbox{10cm}{#1}\par} \makeatother

\newcommand{\abs}[1]{\left\lvert#1\right\rvert}
\newcommand{\norm}[1]{\left\lVert#1\right\rVert}

\newcommand{\wh}[1]{\widehat{#1}}
\newcommand{\wt}[1]{\widetilde{#1}}
\newcommand{\ol}[1]{\overline{#1}}
\newcommand{\ul}[1]{\underline{#1}}

\renewcommand*{\Im}[1]{\operatorname{\mathfrak{Im}}{\left( #1 \right)}}
\renewcommand*{\Re}[1]{\operatorname{\mathfrak{Re}}{\left( #1 \right)}}

\def \bbR     {{\mathbb R}}

\def \cF     {{\mathcal F}}

\def \cN     {{\mathcal N}}

\def \cR     {{\mathcal R}}
\def \cS     {{\mathcal S}}

\def \cW     {{\mathcal W}}

\def \tf     {{\tilde{f}}}

\def \tT     {{\wt{T}}}
\def \tW     {{\wt{W}}}

\newcommand{\set}[1]{\left\{#1\right\}}

\newcommand{\beq}{\begin{equation}}
\newcommand{\eeq}{\end{equation}}
\newcommand{\bal}{\begin{align}}
\newcommand{\eal}{\end{align}}

\ifx \nodefthm \undefined
  \usepackage{amsthm}
  \theoremstyle{plain}
  \newtheorem{thm}{Theorem}[section]
  \newtheorem{lem}[thm]{Lemma}
  \newtheorem{prop}[thm]{Proposition}
  \newtheorem{cor}[thm]{Corollary}
  \theoremstyle{definition}
  \newtheorem{defn}{Definition}[section]

  \newtheorem{exmp}{Example}[section]
  \theoremstyle{remark}
  \newtheorem*{rem}{Remark}
  \newtheorem*{note}{Note}
  \newtheorem{case}{Case}
\else
  \ifx \newtheorem \undefined
  \else
    \newtheorem{thm}[theorem]{Theorem}

    \newtheorem{cor}[theorem]{Corollary}
    \newtheorem{defn}[definition]{Definition}

  \fi
\fi

\begin{document}


\title{The Synchrosqueezing algorithm for time-varying spectral analysis: robustness properties and new paleoclimate applications}


\author{
  Gaurav~Thakur,~%
  Eugene~Brevdo,~%
  Neven~S.~Fu\v{c}kar,~%
  and~Hau-Tieng~Wu~%
\thanks{G. Thakur is with the MITRE Corporation, McLean, VA, e-mail: gthakur@alumni.princeton.edu.}%
\thanks{E. Brevdo is with the Climate Corporation, San Francisco, CA, e-mail:  ebrevdo@math.princeton.edu.}%
\thanks{N.S. Fu\v{c}kar is with the International Pacific Research Center, University of Hawaii, Honolulu, HI, e-mail: nevensf@gmail.com.}%
\thanks{H.-T. Wu is with the Program in Applied and Computational Mathematics, Princeton University, Princeton, NJ, e-mail: hauwu@math.princeton.edu.}%
}

\maketitle


\begin{abstract}
We analyze the stability properties of the Synchrosqueezing transform, a time-frequency signal analysis method that can identify and extract
oscillatory components with time-varying frequency and amplitude. We show that Synchrosqueezing is robust to bounded perturbations
of the signal and to Gaussian white noise. These results justify its applicability to noisy or nonuniformly sampled data that is ubiquitous
in engineering and the natural sciences. We also describe a practical implementation of Synchrosqueezing and provide guidance on tuning its main
parameters. As a case study in the geosciences, we examine characteristics of a key paleoclimate change in the last 2.5 million years, where
Synchrosqueezing provides significantly improved insights.
\end{abstract}


\section[section]{Introduction}

Synchrosqueezing is a time-frequency signal analysis algorithm designed to decompose signals into constituent
components with time-varying oscillatory characteristics. Such signals $f(t)$ have the general form
\begin{equation}\label{eq:sig}
  f(t) \,=\, \sum_{k=1}^K f_k(t) + e(t),
\end{equation}
where each component $f_k(t) = A_k(t)\cos(2\pi \phi_k(t))$ is a Fourier-like oscillatory mode, possibly with time-varying amplitude and
frequency, and $e(t)$ represents noise or measurement error.  The goal is to recover the amplitude $A_k(t)$ at the
instantaneous frequency (IF) $\phi'_k(t)$ for each $k$. \\

Signals of the form \eqref{eq:sig} arise naturally in numerous scientific and engineering applications, where it is often important
to understand their time-varying spectral properties. Many time-frequency (TF) transforms exist to analyze such signals, such as
the short-time Fourier transform (STFT), continuous wavelet transform (CWT), and the Wigner-Ville distribution (WVD)
\cite{Flandrin1999, Allen1977, Claasen1980, Rioul1992}. Synchrosqueezing is related to the class of time-frequency reassignment
(TFR) algorithms, used in the estimation of IFs from the modulus of a TF representation. TFR methods originate from a study of the
STFT, which "smears" the energy of the superimposed IFs around their center frequencies in the spectrogram. TFR methods apply a
post-processing ``reassignment'' map that focuses the spectrogram's energy towards the IF curves and results in a sharpened TF plot.
However, standard TFR methods do not allow for reconstruction (synthesis) of the components $f_k(t)$. \cite{Flandrin1999, Fulop2006, auger_flandrin:1995} \\

Originally introduced in the context of audio signal analysis \cite{Daubechies1996}, Synchrosqueezing was recently studied further in
\cite{Daubechies2010} and shown to be an alternative to the {\em Empirical Mode Decomposition} (EMD) method \cite{Huang1998}
with a more firm theoretical foundation. EMD has been found to be a useful tool for analyzing and decomposing natural signals and,
like EMD, Synchrosqueezing can extract and delineate components with time-varying spectrum. Furthermore, like EMD, and unlike classical
TFR techniques, it allows for the reconstruction of these components. Synchrosqueezing can be adapted to work "on top of" many of the
classical TF transforms. In this paper, we focus on the original, CWT-based approach studied in \cite{Daubechies1996} and
\cite{Daubechies2010}, although an STFT-based alternative was developed in \cite{Thakur2010} and other variants are also possible. \\

The purpose of this paper is threefold. First, in Section \ref{sec:main}, we study the stability properties of Synchrosqueezing.
We build on the theory presented in \cite{Daubechies2010} and prove that Synchrosqueezing is stable under bounded, deterministic
perturbations in the signal as well as under corruption by Gaussian white noise. This justifies the use of the algorithm in real-world
cases where different sources of error are present, such as thermal noise incurred from signal acquisition or quantization and
interpolation errors in processing the data.\\

Second, in Section \ref{sec:analysis}, we explain how Synchrosqueezing is implemented in practice and reformulate the approach from
\cite{Daubechies2010} into a discretized form that is more numerically viable and accessible to a wider audience. We also provide practical
guidelines for choosing several parameters that arise in this process. A MATLAB implementation of the algorithm has been developed
and is freely available as part of the Synchrosqueezing Toolbox \cite{SSToolbox}. In Section \ref{sec:wmisc}, we illustrate the
algorithm on several numerical test cases. We study its performance and compare it to some of the well known TF and TFR techniques. \\

Finally, in Section \ref{sec:SSpaleo}, we visit a key question in the Earth's climate of the last 2.5 million years (\,Myr).
We analyze a calculated solar flux index and paleoclimate records of the oxygen isotope ratio $\delta^{18}O$, an index of climate state,
over this period. We demonstrate that Synchrosqueezing clearly delineates the orbital cycles of the solar radiation and provides a greatly
improved representation of the projection of orbital signals in $\delta^{18}O$ records. In comparison to previous spectral analyses of
$\delta^{18}O$ time series, the Synchrosqueezing representation provides more robust and precise estimates in the time-frequency plane,
and contributes to our understanding of the link between solar forcing and climate response on very long time scales (on the order of $10$\,kyr\,-\,$1$\,Myr).


\section{\label{sec:main}The Stability of Synchrosqueezing}

\noindent In this section, we state and prove our main theorems on
the stability properties of Synchrosqueezing. We first review the
existing results on wavelet-based Synchrosqueezing and some associated
notation and terminology from the paper \cite{Daubechies2010}. We
define a class of functions (signals) on which the theory is established.

\noindent \begin{defn}{[}Sums of Intrinsic Mode Type (IMT) Functions{]}
The space $\mathcal{A}_{\epsilon,d}$ of \textit{superpositions of
IMT functions}, with smoothness $\epsilon>0$ and separation $d>0$,
consists of functions having the form $f(t)=\sum_{k=1}^{K}f_{k}(t)$
with $f_{k}(t)=A_{k}(t)e^{2\pi i\phi_{k}(t)}$, where for each $k$,
the $A_{k}$ and $\phi_{k}$ satisfy the following conditions.
\begin{align*}
 & A_{k}\in L^{\infty}\cap C^{1},\quad\phi_{k}\in C^{2},\quad\phi'_{k},\phi''_{k}\in L^{\infty},\quad\inf_{t}\phi'_{k}(t)>0,\\
 & \forall t\quad\left|A_{k}'(t)\right|\leq\epsilon\left|\phi_{k}'(t)\right|,\quad\left|\phi_{k}''(t)\right|\leq\epsilon\left|\phi_{k}'(t)\right|,\quad\mathrm{and}
\end{align*}

\noindent 
\begin{align*}
\frac{\phi'_{k}(t)-\phi'_{k-1}(t)}{\phi'_{k}(t)+\phi'_{k-1}(t)} & \geq d.
\end{align*}
\end{defn}

\noindent Functions in the class $\mathcal{A}_{\epsilon,d}$ are composed
of several Fourier-like oscillatory components with slowly time-varying
amplitudes and sufficiently smooth frequencies. The IF components
$\phi'_{k}$ are strongly separated in the sense that high frequency
components are spaced exponentially further apart than low frequency
ones.\\

We normalize the Fourier transform by $\widehat{h}(\xi)=\int_{-\infty}^{\infty}h(x)e^{-2\pi i\xi x}dx$
and use the notation $\tilde{\epsilon}=\epsilon^{1/3}$. Now for a
given mother wavelet $\psi$, the \textit{continuous wavelet transform}
(CWT) of $f$ at scale $a$ and time shift $b$ is given by $W_{f}(a,b)=a^{-1/2}\int_{-\infty}^{\infty}f(t)\overline{\psi(\frac{t-b}{a})}dt$.
If $\hat{f}$ is supported in $(0,\infty)$, then the inversion of
the CWT can be expressed as $f(b)=\frac{1}{\mathcal{R}_{\psi}}\int_{0}^{\infty}a^{-3/2}W_{f}(a,b)da,$
where we let $\mathcal{R}_{\psi}=\int_{0}^{\infty}\xi^{-1}\overline{\widehat{\psi}(\xi)}d\xi$
\cite[p. 6]{Daubechies2010}. We use the CWT to define the \textit{phase
transform $\omega_{f}(a,b)$} by
\begin{equation}
\omega_{f}(a,b)=\frac{\partial_{t}W_{f}(a,b)}{2\pi iW_{f}(a,b)}.\label{Omega}
\end{equation}
$\omega_{f}(a,b)$ can be thought of as an ``FM demodulated'' frequency
estimate that cancels out the influence of the wavelet $\psi$ on
$W_{f}(a,b)$ and results in a modified time-scale representation
of $f$. We can use this to consider the following operator.

\noindent \begin{defn}{[}CWT Synchrosqueezing{]} Let $f\in\mathcal{A}_{\epsilon,d}$
and $h\in C_{0}^{\infty}$ be a smooth function such that $\norm{h}_{L^{1}}=1$.
The \textit{Wavelet Synchrosqueezing transform }with accuracy $\delta$
and thresholds $\tilde{\epsilon}$ and $M$ is defined by
\begin{equation}
S_{f,\tilde{\epsilon}}^{\delta,M}(b,\eta)=\int_{\Gamma_{f,\tilde{\epsilon}}^{M}}\frac{W_{f}(a,b)}{a^{3/2}}\frac{1}{\delta}h\left(\frac{\eta-\omega_{f}(a,b)}{\delta}\right)da,\label{SS}
\end{equation}
where $\Gamma_{f,\tilde{\epsilon}}^{M}=\set{(a,b):a\in[M^{-1},M],|W_{f}(a,b)|>\tilde{\epsilon}}$.
We also denote $S_{f,\tilde{\epsilon}}^{\delta}(b,\eta):=S_{f,\tilde{\epsilon}}^{\delta,\infty}(b,\eta)$,
with the condition $a\in[M^{-1},M]$ replaced by $a>0$.\end{defn}

\noindent For sufficiently small $\delta$, this operator can be thought
of as a partial inversion of the CWT of $f$ (over the scale $a$),
but only taken over small bands around level curves in the time-scale
plane (where $\omega_{f}(a,b)\approx\eta$) and ignoring the rest
of the plane. As we let $\delta\to0$, the domain of the inversion
becomes concentrated on the level curves $\{(a,b):\omega_{f}(a,b)=\eta\}$.
The idea is that this localization process will allow us to recover
the components $f_{k}$ more accurately than inverting the CWT over
the entire time-scale plane. The following theorem was the main result
of \cite{Daubechies2010}.

\begin{thm} (Daubechies, Lu, Wu) \label{SSThm} Let $f=\sum_{k=1}^{K}A_{k}e^{2\pi i\phi_{k}}\in\mathcal{A}_{\epsilon,d}$
and $\tilde{\epsilon}=\epsilon^{1/3}$. Pick a function $h\in C_{0}^{\infty}$
with $\left\Vert h\right\Vert _{L^{1}}=1$, and pick a wavelet $\psi\in C^{1}$
such that its Fourier transform $\widehat{\psi}$ is supported in
$[1-\Delta,1+\Delta]$ for some $\Delta<\frac{d}{1+d}$. Then the
following statements hold for each $k$:
\begin{enumerate}
\item Define the ``scale band'' $Z_{k}=\{(a,b):|a\phi_{k}'(b)-1|<\Delta\}$.
For each point $(a,b)\in Z_{k}$ with $|W_{f}(a,b)|>\tilde{\epsilon}$,
we have
\[
|\omega_{f}(a,b)-\phi_{k}'(b)|\leq\tilde{\epsilon},
\]
and if $(a,b)\not\in Z_{k}$ for any $k$, then $|W_{f}(a,b)|\leq\tilde{\epsilon}$.
\item There is a constant $C_{1}$ such that for all $b\in\mathbb{R}$,
\[
\left|\lim_{\delta\rightarrow0}\left(\frac{1}{\mathcal{R}_{\psi}}\int_{\{\eta:|\eta-\phi'_{k}(b)|\leq\tilde{\epsilon}\}}S_{f,\tilde{\epsilon}}^{\delta}(b,\eta)d\eta\right)-A_{k}(b)e^{2\pi i\phi_{k}(b)}\right|\leq C_{1}\tilde{\epsilon}.
\]

\end{enumerate}
\end{thm}

This result shows how Synchrosqueezing can identify and extract the
components $\set{f_{k}}$ from $f$. The first part of Theorem \ref{SSThm}
says that the plot of $|S_{f,\tilde{\epsilon}}^{\delta}|$ is concentrated
around the instantaneous frequency curves $\{\phi_{k}'\}$. The second
part of Theorem \ref{SSThm} tells us that we can reconstruct each
component $f_{k}$ by completing the inversion of the CWT, locally
over small frequency bands surrounding $\phi_{k}'$. In particular,
it implies that we can recover the amplitudes $A_{k}$ by taking absolute
values. Theorem \ref{SSThm} also suggests that components $f_{k}$
of small magnitude may be difficult to detect (as their CWTs become
smaller than $\tilde{\epsilon}$). \\

We can now state our new results on the robustness properties of Synchrosqueezing.
The following theorem shows that the results in Theorem \ref{SSThm}
essentially still hold if we perturb $f$ by a small (deterministic)
error term $e$.

\begin{thm}\label{SSStableThm}Let $f\in\mathcal{A}_{\epsilon,d}$
and suppose we have a corresponding $\epsilon$, $h$, $\psi$ and
$\Delta$ as given in Theorem \ref{SSThm}. Suppose that $g=f+e$,
where $e\in L^{\infty}$ is a small error term that satisfies $\left\Vert e\right\Vert _{L^{\infty}}\leq\epsilon/\max(\left\Vert \psi\right\Vert _{L^{1}},\left\Vert \psi'\right\Vert _{L^{1}})$.
For each $k$, let $M_{k}\geq1$ be the ``maximal frequency range''
given by

$M_{k}=$$\max\left(\frac{1}{1-\Delta}\left\Vert \phi_{k}'\right\Vert _{L^{\infty}},(1+\Delta)\left\Vert \frac{1}{\phi_{k}'}\right\Vert _{L^{\infty}}\right)$.
Then the following statements hold for each $k$:
\begin{enumerate}
\item Assume $a\in[M_{k}^{-1},M_{k}]$. For each point $(a,b)\in Z_{k}$
with $|W_{g}(a,b)|>M_{k}^{1/2}\epsilon+\tilde{\epsilon}$, we have
\[
|\omega_{g}(a,b)-\phi_{k}'(b)|\leq C_{2}\tilde{\epsilon}
\]
for some constant $C_{2}=O(M_{k})$. If $(a,b)\not\in Z_{k}$ for
any $k$, then $|W_{g}(a,b)|\leq M_{k}^{1/2}\epsilon+\tilde{\epsilon}$. 
\item There is a constant $C_{3}=O(M_{k})$ such that for all $b\in\mathbb{R}$,
\[
\left|\lim_{\delta\rightarrow0}\left(\frac{1}{\mathcal{R}_{\psi}}\int_{\{\eta:|\eta-\phi'_{k}(b)|\leq C_{2}\tilde{\epsilon}\}}S_{g,M_{k}^{1/2}\epsilon+\tilde{\epsilon}}^{\delta,M_{k}}(b,\eta)d\eta\right)-A_{k}(b)e^{2\pi i\phi_{k}(b)}\right|\leq C_{3}\tilde{\epsilon}.
\]

\end{enumerate}
\end{thm}

\noindent \begin{proof} It is clear that
\begin{align}
|W_{f}(a,b)-W_{g}(a,b)|\leq\left\Vert f-g\right\Vert _{L^{\infty}}a^{1/2}\int_{-\infty}^{\infty}\left|\overline{\psi\left(t-\frac{b}{a}\right)}\right|dt\leq a^{1/2}\epsilon.\label{wfdiff}
\end{align}
Similarly, we also have $|\partial_{b}W_{f}(a,b)-\partial_{b}W_{g}(a,b)|\leq a^{-1/2}\epsilon$.
Now if $(a,b)\not\in Z_{k}$ for any $k$, then using Thm. \ref{SSThm}
gives
\begin{align}
|W_{g}(a,b)| & \leq|W_{g}(a,b)-W_{f}(a,b)|+|W_{f}(a,b)|\leq M_{k}^{1/2}\epsilon+\tilde{\epsilon}.\label{OutsideZk}
\end{align}
On the other hand, if for some $k$, $(a,b)\in Z_{k}$ and $|W_{g}(a,b)|>M_{k}^{1/2}\epsilon+\tilde{\epsilon}$,
then by (\ref{wfdiff}) and Thm. \ref{SSThm},
\begin{align}
|\omega_{g}(a,b)-\phi_{k}'(b)| & \leq|\omega_{g}(a,b)-\omega_{f}(a,b)|+|\omega_{f}(a,b)-\phi_{k}'(b)|\nonumber \\
 & \leq\left|\frac{W_{g}(a,b)-W_{f}(a,b)}{W_{g}(a,b)W_{f}(a,b)}\partial_{b}W_{f}(a,b)+\frac{\partial_{b}W_{f}(a,b)-\partial_{b}W_{g}(a,b)}{W_{g}(a,b)}\right|+\tilde{\epsilon}\nonumber \\
 & \leq\frac{M_{k}^{1/2}\epsilon}{(M_{k}^{1/2}\epsilon+\tilde{\epsilon})\tilde{\epsilon}}\left(M_{k}^{1/2}\left\Vert f\right\Vert _{L^{\infty}}\left\Vert \psi'\right\Vert _{L^{1}}\right)+\frac{M_{k}^{1/2}\epsilon}{M_{k}^{1/2}\epsilon+\tilde{\epsilon}}+\tilde{\epsilon}\nonumber \\
 & \leq C_{2}\tilde{\epsilon},\label{OmegaBound}
\end{align}
where $C_{2}$ depends only on $f$, $\psi$ and $M_{k}$. For the
second part of Thm. \ref{SSStableThm}, we fix $k$ and $b$ and use
the following calculation \cite[p.12]{Daubechies2010}:
\begin{align}
\lim_{\delta\rightarrow0} & \int_{|\eta-\phi_{k}'(b)|\leq\tilde{\epsilon}}S_{f,\tilde{\epsilon}}^{\delta}(b,\eta)d\eta=\int_{\mathrm{D}(b,f,\tilde{\epsilon},\tilde{\epsilon},\infty)}a^{-3/2}W_{f}(a,b)da,\label{SSPart1}
\end{align}
where
\begin{equation}
\mathrm{D}(b,f,\epsilon_{1},\epsilon_{2},M):=\{a:|W_{f}(a,b)|>\epsilon_{1},|\omega_{f}(a,b)-\phi_{k}'(b)|\leq\epsilon_{2},a\in[M^{-1},M]\}.\label{DSet}
\end{equation}

\noindent It is also shown in \cite[p. 12]{Daubechies2010} that if
$a\in\mathrm{D}(b,f,\tilde{\epsilon},\tilde{\epsilon},\infty)$, then
$(a,b)\in Z_{k}$, so $M_{k}^{-1}\leq a\leq M_{k}$. This means that
in (\ref{SSPart1}), we can replace $S_{f,\tilde{\epsilon}}^{\delta}(b,\eta)$
by $S_{f,\tilde{\epsilon}}^{\delta,M_{k}}(b,\eta)$ and $\mathrm{D}(b,f,\tilde{\epsilon},\tilde{\epsilon},\infty)$
by $\mathrm{D}(b,f,\tilde{\epsilon},\tilde{\epsilon},M_{k})$. We
can also get a result identical to (\ref{SSPart1}) for $g$ by simply
repeating the argument in \cite{Daubechies2010}. First, note that
as $\delta\to0$, the expression
\begin{equation}
\int_{|\eta-\phi'_{k}(b)|\leq C_{2}\tilde{\epsilon}}a^{-3/2}W_{g}(a,b)\frac{1}{\delta}h\left(\frac{\eta-\omega_{g}(a,b)}{\delta}\right)d\eta\label{ApproxId}
\end{equation}
converges to $a^{-3/2}W_{g}(a,b)\chi_{\set{|\omega_{g}(a,b)-\phi'_{k}(b)|<C_{2}\tilde{\epsilon}}}(a)$
for almost all $a\in[M_{k}^{-1},M_{k}]$, where $\chi$ is the characteristic
function of a set. This shows that
\begin{align}
 & \lim_{\delta\rightarrow0}\int_{|\eta-\phi_{k}'(b)|\leq C_{2}\tilde{\epsilon}}S_{g,M_{k}^{1/2}\epsilon+\tilde{\epsilon}}^{\delta,M_{k}}(b,\eta)d\eta\nonumber \\
= & \int_{(a,b)\in\Gamma_{g,M_{k}^{1/2}\epsilon+\tilde{\epsilon}}^{M_{k}}}\lim_{\delta\to0}\int_{|\eta-\phi'_{k}(b)|\leq C_{2}\tilde{\epsilon}}a^{-3/2}W_{g}(a,b)\frac{1}{\delta}h\left(\frac{\eta-\omega_{g}(a,b)}{\delta}\right)d\eta da\label{TermChange}\\
= & \int_{\mathrm{D}(b,g,M_{k}^{1/2}\epsilon+\tilde{\epsilon},C_{2}\tilde{\epsilon},M_{k})}a^{-3/2}W_{g}(a,b)da.\label{SSPart2}
\end{align}
We can justify exchanging the order of integrations and limits in
(\ref{TermChange}) by the Fubini and dominated convergence theorems,
since (\ref{ApproxId}) is bounded by $|a^{-3/2}W_{g}(a,b)|\in L^{1}(\{a:|W_{g}(a,b)|>M_{k}^{1/2}\epsilon+\tilde{\epsilon},a\in[M_{k}^{-1},M_{k}]\})$
for all $\delta$. We also note that (\ref{wfdiff}) and (\ref{OmegaBound})
show that in the set $\mathrm{D}(b,f,\tilde{\epsilon},\tilde{\epsilon},M_{k})\backslash\mathrm{D}(b,g,M_{k}^{1/2}\epsilon+\tilde{\epsilon},C_{2}\tilde{\epsilon},M_{k})$,
we have $|W_{f}(a,b)|\leq2M_{k}^{1/2}\epsilon+\tilde{\epsilon}$.
We can now use the result of Thm. \ref{SSThm} along with (\ref{OutsideZk}),
(\ref{SSPart1}) and (\ref{SSPart2}) to find that
\begin{align*}
 & \left|\lim_{\delta\rightarrow0}\int_{|\eta-\phi_{k}'(b)|\leq\tilde{\epsilon}}S_{f,\tilde{\epsilon}}^{\delta,M_{k}}(b,\eta)d\eta-\lim_{\delta\rightarrow0}\int_{|\eta-\phi_{k}'(b)|\leq C_{2}\tilde{\epsilon}}S_{g,M_{k}^{1/2}\epsilon+\tilde{\epsilon}}^{\delta,M_{k}}(b,\eta)d\eta\right|\\
= & \left|\int_{\mathrm{D}(b,f,\tilde{\epsilon},\tilde{\epsilon},M_{k})}a^{-3/2}W_{f}(a,b)-\int_{\mathrm{D}(b,g,M_{k}^{1/2}\epsilon+\tilde{\epsilon},C_{2}\tilde{\epsilon},M_{k})}a^{-3/2}W_{g}(a,b)da\right|\\
\leq & \int_{\mathrm{D}(b,g,M_{k}^{1/2}\epsilon+\tilde{\epsilon},C_{2}\tilde{\epsilon},M_{k})}\left|a^{-3/2}(W_{f}(a,b)-W_{g}(a,b))\right|da\\
 & \quad+\int_{\mathrm{D}(b,f,\tilde{\epsilon},\tilde{\epsilon},M_{k})\backslash\mathrm{D}(b,g,M_{k}^{1/2}\epsilon+\tilde{\epsilon},C_{2}\tilde{\epsilon},M_{k})}\left|a^{-3/2}W_{f}(a,b)\right|da\\
 & \leq\int_{M_{k}^{-1}}^{M_{k}}a^{-1}\epsilon da+\int_{M_{k}^{-1}}^{M_{k}}a^{-3/2}\left(2M_{k}^{1/2}\epsilon+\tilde{\epsilon}\right)da\\
 & \leq(2\log M_{k})\epsilon+2\left(M_{k}^{1/2}-M_{k}^{-1/2}\right)\left(2M_{k}^{1/2}\epsilon+\tilde{\epsilon}\right)\\
 & \leq C_{3}\tilde{\epsilon}.
\end{align*}
Combining this with the result of Thm. \ref{SSThm} finishes the proof.

\noindent \end{proof}

Thm. \ref{SSStableThm} shows that each component $f_{k}$ can be
recovered with an accuracy proportional to the perturbation $e$ and
its maximal frequency range $M_{k}$, with mid-range IFs ($M_{k}$
close to $1$) resulting in the best estimates. In addition, Thm.
\ref{SSStableThm} implies that we can replace a continuous-time function
$f$ with discrete approximations of it. In many applications, we
only have a collection of samples $\{f(t_{n})\}$ available instead
of the whole function $f$, where $\{t_{n}\}$ is a sequence of (possibly
nonuniformly spaced) sampling points. We can address this situation
in the following way.

\begin{cor}\label{cor:splinestable} Let $f_{s}\in C^{2}$ be the
cubic spline interpolant formed from $\{f(t_{n})\}$ and define ${\displaystyle \Lambda=\sup_{n}|t'_{n+1}-t'_{n}|}$.
Then the errors in the estimating $\phi_{k}'(b)$ and $f_{k}(b)$
from $f_{s}$ are both $O(M_{k}\Lambda^{4/3})$ for all $b$.\end{cor}

\noindent \begin{proof}This follows from Thm. \ref{SSStableThm}
and the following standard estimate on cubic spline approximations
\cite[p. 97]{Stewart1998}:

\begin{center}
$\left\Vert f_{s}-f\right\Vert _{L^{\infty}}\leq\frac{5}{384}\Lambda^{4}\|f^{(4)}\|_{L^{\infty}}$ 
\par\end{center}

\end{proof}

\noindent This means that we can work with the spline $f_{s}$ instead
of $f$, and as long as the minimum sampling rate $\Lambda^{-1}$
is high enough, the results will be close. In practice, we find that
the errors are localized in time to areas of low sampling rate, low
component amplitude, and/or high component frequency (see, e.g., \S\ref{sec:wmisc}).\\

The second result of this paper is that Sychrosqueezing is also robust
to additive Gaussian white noise. We start by defining Gaussian white
noise in continuous-time. Let $\mathcal{S}$ be the Schwartz
class of smooth functions with rapid decay (see \cite{K96}). A (real)
stationary {\em generalized Gaussian process} $G$ is a random
linear functional on $\mathcal{S}$ such that all finite collections
$\{G(f_{i})\}$ with $f_{i}\in\mathcal{S}$ are jointly Gaussian variables
and have the same distribution for all translates of $f_{i}$. Such
a process is characterized by a mean functional $\mathbb{E}(G(f_{1}))=T(f_{1})$
and a covariance functional $\mathbb{E}((G(f_{1})-T(f_{1}))\overline{(G(f_{2})-T(f_{2}))})=\left\langle f_{1},Rf_{2}\right\rangle $
for some operators $T:\mathcal{S}\to\mathcal{S}$ and $R:\mathcal{S}\to\mathcal{S}$,
where $\left\langle f_{1},f_{2}\right\rangle =\int_{-\infty}^{\infty}f_{1}(t)\overline{f_{2}(t)}dt$
is the $L^{2}$ inner product. {\em Gaussian white noise} $N$
with power $\sigma^{2}$ is such a process with $T=0$
and $R=\sigma^{2}I$, where $I$ is the identity operator. We refer
to \cite{K96} for more details on these concepts and to \cite{Ga08}
for basic facts on complex Gaussian variables that are used below.

\begin{thm}\label{SSStableThm2}Let $f\in\mathcal{A}_{\epsilon,d}$
and suppose we have a corresponding $\epsilon$, $h$, $\psi$, $\Delta$
and $M_{k}$ as given in Thm \ref{SSThm} and \ref{SSStableThm},
with the additional assumptions that $\psi\in\mathcal{S}$ and $\left|\left\langle \psi,\psi'\right\rangle \right|<\|\psi\|_{L^{2}}\|\psi'\|_{L^{2}}$.
Let $g=f+N$, where $N$ is Gaussian white noise with spectral density
$\epsilon^{2+p}$ for some $p>0$. Then the following statements hold
for each $k$:
\begin{enumerate}
\item Assume $a\in[M_{k}^{-1},M_{k}]$. For each point $(a,b)\in Z_{k}$
with $|W_{f}(a,b)|>\tilde{\epsilon}$, there are constants $E_1$ and
$C_{2}$' such that with probability $1-e^{-E_1\epsilon^{-p}}$,
\[
|\omega_{g}(a,b)-\phi_{k}'(b)|\leq C_{2}'\tilde{\epsilon}.
\]
If $(a,b)\not\in Z_{k}$ for any $k$, then with probability $1-e^{-E_2\epsilon^{-p}}$ for some constant $E_2$,
$|W_{g}(a,b)|\leq\tilde{\epsilon}+\frac{1}{2}\epsilon$.
\item There is a constant $C_{3}'$ such that with probability $1-e^{-E_1\epsilon^{-p}}$,
we have for all $b\in\mathbb{R}$ that
\[
\left|\lim_{\delta\rightarrow0}\left(\frac{1}{\mathcal{R}_{\psi}}\int\limits _{\{\eta:|\eta-\phi'_{k}(b)|\leq C_{2}'\tilde{\epsilon}\}}S_{g,M_{k}^{1/2}\epsilon+\tilde{\epsilon}}^{\delta,M_{k}}(b,\eta)d\eta\right)-A_{k}(b)e^{2\pi i\phi_{k}(b)}\right|\leq C_{3}'\tilde{\epsilon}.
\]

\end{enumerate}
\noindent \end{thm}\begin{proof} The CWT of $g$, $W_{g}(a,b)$,
is understood as the Gaussian variable $W_{f}(a,b)+\overline{N(\psi_{a,b})}$,
where $\psi_{a,b}(x)=a^{-1/2}\psi\left(\frac{x-b}{a}\right)$. We
have $\mathbb{E}(N(\psi_{a,b}))=0$,
\[
\mathbb{E}(N(\psi_{a,b})\overline{N(\psi_{a,b})})=\frac{\epsilon^{2+p}}{a}\int_{-\infty}^{\infty}\psi\left(\frac{x}{a}\right)\overline{\psi\left(\frac{x}{a}\right)}dx=\epsilon^{2+p}\left\langle \psi,\psi\right\rangle =\epsilon^{2+p}\|\psi\|_{L^{2}}^{2},
\]
and since $\mathrm{supp}(\widehat{\psi})$ is positive,
\[
\mathbb{E}(N(\psi_{a,b})^{2})=\epsilon^{2+p}\left\langle \psi,\overline{\psi}\right\rangle =\epsilon^{2+p}\int_{-\infty}^{\infty}\hat{\psi}(\xi)\hat{\psi}(-\xi)d\xi=0.
\]
Similarly, $\partial_{b}W_{g}(a,b)$ is the random variable $\partial_{b}W_{f}(a,b)+\overline{N(\psi_{a,b}')}$,
where $\psi{}_{a,b}'(x)=a^{-3/2}\psi'\left(\frac{x-b}{a}\right)$.
By the same arguments as before and noting that $\mathrm{supp}(\widehat{\psi'})\subset\mathrm{supp}(\widehat{\psi})$,
we obtain the formulas:
\begin{align*}
\mathbb{E}(N(\psi_{a,b}')) & =\mathbb{E}(N(\psi_{a,b}')^{2})=\mathbb{E}(N(\psi_{a,b})N(\psi_{a,b}'))=0\\
\mathbb{E}(N(\psi_{a,b}')\overline{N(\psi_{a,b}')}) & =\epsilon^{2+p}a^{-2}\left\Vert \psi'\right\Vert _{L^{2}}^{2}\\
\mathbb{E}(N(\psi_{a,b})\overline{N(\psi_{a,b}')}) & =\epsilon^{2+p}a^{-1}\left\langle \psi,\psi'\right\rangle .
\end{align*}
This shows that the Gaussian variables $N(\psi_{a,b})$ and $(N(\psi_{a,b}),N(\psi_{a,b}'))\in\mathbb{C}^{2}$
have zero pseudo-covariance matrices, so they are circularly symmetric.
If we define the matrix
\[
V=\left(\begin{array}{cc}
\left\Vert \psi\right\Vert _{L^{2}}^{2} & a^{-1}\left\langle \psi,\psi'\right\rangle \\
a^{-1}\left\langle \psi',\psi\right\rangle  & a^{-2}\left\Vert \psi'\right\Vert _{L^{2}}^{2}
\end{array}\right),
\]
then the distribution of $(N(\psi_{a,b}),N(\psi_{a,b}'))$ is given
by
\[
\frac{e^{-\frac{1}{\epsilon^{2+p}}\overline{(w,z)}\cdot V^{-1}(w,z)}}{\pi^{2}\epsilon^{4+2p}\det V}dwdz.
\]
Since $V$ is invertible and self-adjoint, we can write $V^{-1}=U^{*}DU$,
where $D$ is diagonal and $U$ is unitary. We have $D_{11}D_{22}=\det(V^{-1})=\det(V)^{-1}$
and $D_{11}+D_{22}=\mathrm{trace}(V^{-1})=(\left\Vert \psi\right\Vert _{L^{2}}^{2}+a^{-2}\left\Vert \psi'\right\Vert _{L^{2}}^{2})\det(V)^{-1}$.\\

\noindent For a point $(a,b)$, we now define the events $G_{1}=\{|N(\psi_{a,b})|<\frac{\epsilon}{2}\}$,
$G_{2}=\{|N(\psi_{a,b}')|<\frac{\epsilon}{2}\}$ and $H_{k}=\{|\omega_{g}(a,b)-\phi_{k}'(b)|\leq C_{2}'\tilde{\epsilon}\}$
for each $k$. We want to estimate $P(G_{1})$ and $P\left(G_{1}\cap G_{2}\right)$.
Using the above calculations and taking $E_2 = \frac{1}{4}\|\psi\|_{L^{2}}^{-2}$, we find that
\begin{eqnarray*}
P(G_{1}) & = & \frac{1}{\pi\epsilon^{2+p}\|\psi\|_{L^{2}}^{2}}\int_{|z|<\frac{\epsilon}{2}}e^{-\frac{|z|^{2}}{\epsilon^{2+p}}\|\psi\|_{L^{2}}^{-2}}dz\\
 & = & \frac{2}{\epsilon^{2+p}\|\psi\|_{L^{2}}^{2}}\int_{0}^{\epsilon/2}re^{-\frac{r^{2}}{\epsilon^{2+p}}\|\psi\|_{L^{2}}^{-2}}dr\\
 & = & 2\int_{0}^{(4\epsilon^{p}\|\psi\|_{L^{2}}^{2})^{-1/2}}re^{-r^{2}}dr\\
 & = & 1-e^{-E_2 \epsilon^{-p}}.
\end{eqnarray*}
Let $E_1=\min_{a\in[M_{k}^{-1},M_{k}]}\frac{1}{8}\left(D_{11}+D_{22}\right)>0$.
We note that any rotated polydisk of radius $r$ in $(w,z)\in\mathbb{C}^{2}$ contains a smaller polydisk of radius $2^{-1/2}r$ that is aligned with the $w$ and $z$ planes,
and use the transformation $(w',z')=U(w,z)$ to estimate

\begin{eqnarray*}
P(G_{1}\cap G_{2}) & = & \int_{\{|w|<\frac{\epsilon}{2},|z|<\frac{\epsilon}{2}\}}\frac{e^{-\frac{1}{\epsilon^{2+p}}\overline{(w,z)}\cdot V^{-1}(w,z)}}{\pi^{2}\epsilon^{4+2p}\det V}dwdz\\
 & = & \int_{\{|(0,1)\cdot U^{*}(w',z')|<\frac{\epsilon}{2},|(1,0)\cdot U^{*}(w',z')|<\frac{\epsilon}{2}\}}\frac{e^{-\frac{1}{\epsilon^{2+p}}\left(D_{11}|w'|^{2}+D_{22}|z'|^{2}\right)}}{\pi^{2}\epsilon^{4+2p}\det V}dw'dz'\\
 & \geq & \int_{\{|w'|^{2}+|z'|^{2}<\frac{\epsilon^{2}}{4}\}}\frac{e^{-\frac{1}{\epsilon^{2+p}}\left(D_{11}|w'|^{2}+D_{22}|z'|^{2}\right)}}{\pi^{2}\epsilon^{4+2p}\det V}dw'dz'\\
 & \geq & \int_{\{|z'|<2^{-3/2}\epsilon,|w'|<2^{-3/2}\epsilon\}}\frac{e^{-\frac{1}{\epsilon^{2+p}}\left(D_{11}|w'|^{2}+D_{22}|z'|^{2}\right)}}{\pi^{2}\epsilon^{4+2p}\det V}dw'dz'\\
 & = & \frac{4}{D_{11}D_{22}\det V}\int_{0}^{(8\epsilon^{p}D_{22}^{-1})^{-1/2}}\int_{0}^{(8\epsilon^{p}D_{11}^{-1})^{-1/2}}re^{-r^{2}}se^{-s^{2}}drds\\
 & = & \left(1-e^{-\frac{1}{8}\epsilon^{-p}D_{11}}\right)\left(1-e^{-\frac{1}{8}\epsilon^{-p}D_{22}}\right)\\
 & > & 1-e^{-E_1\epsilon^{-p}}.
\end{eqnarray*}
Now let $C_{2}'=2M_{k}^{1/2}\left\Vert f\right\Vert _{L^{\infty}}\left\Vert \psi'\right\Vert _{L^{1}}+3$.
If $(a,b)\not\in Z_{k}$ for any $k$, then Theorem \ref{SSThm} shows
that $G_{1}$ implies $|W_{g}(a,b)|<\tilde{\epsilon}+\frac{1}{2}\epsilon$.
Conversely, if $(a,b)\in Z_{k}$ for some $k$, we follow the same
arguments as in Theorem \ref{SSStableThm} to find that
\begin{align*}
P\left(H_{k}\right)\geq & P\left(H_{k}|G_{1}\cap G_{2}\right)P\left(G_{1}\cap G_{2}\right)\\
\geq & P\bigg(\frac{1}{|W_{g}(a,b)|}\left|\frac{\partial_{b}W_{f}(a,b)}{W_{f}(a,b)}(W_{g}(a,b)-W_{f}(a,b))-(\partial_{b}W_{g}(a,b)-\partial_{b}W_{f}(a,b))\right|\\
 & +\tilde{\epsilon}\leq C_{2}'\tilde{\epsilon}\bigg|G_{1}\cap G_{2}\bigg)P\left(G_{1}\cap G_{2}\right)\\
\geq & P\left(\frac{1}{\tilde{\epsilon}-\frac{1}{2}\epsilon}\left(\left|\frac{\partial_{b}W_{f}(a,b)}{W_{f}(a,b)}\right|\epsilon+\epsilon\right)+\tilde{\epsilon}\leq C_{2}'\tilde{\epsilon}\right)P\left(G_{1}\cap G_{2}\right)\\
\geq & P\left(2M_{k}^{1/2}\left\Vert f\right\Vert _{L^{\infty}}\left\Vert \psi'\right\Vert _{L^{1}}+3\leq C_{2}'\right)P\left(G_{1}\cap G_{2}\right)\\
= & P\left(G_{1}\cap G_{2}\right).
\end{align*}
The second statement in Theorem \ref{SSStableThm2} can be shown in
an analogous way. Let $C_{3}'=2M_{k}^{1/2}((M_{k}^{1/2}+1)\tilde{\epsilon}^{2}+1)+C_{1}$
and recall the definition (\ref{DSet}). We fix $k$ and use the above
result to estimate
\begin{align*}
 & P\left(\left|\lim_{\delta\rightarrow0}\int{}_{|\eta-\phi_{k}'(b)|\leq C_{2}'\tilde{\epsilon}}S_{g,\tilde{\epsilon}-\frac{1}{2}\epsilon}^{\delta,M_{k}}(b,\eta)d\eta-A_{k}(b)e^{i\phi_{k}(b)}\right|<C_{3}'\tilde{\epsilon}\right)\\
\geq & P\bigg(\left|\lim_{\delta\rightarrow0}\int{}_{|\eta-\phi_{k}'(b)|\leq\tilde{\epsilon}}S_{f,\tilde{\epsilon}}^{\delta,M_{k}}(b,\eta)d\eta-\lim_{\delta\rightarrow0}\int{}_{|\eta-\phi_{k}'(b)|\leq C_{2}'\tilde{\epsilon}}S_{g,\tilde{\epsilon}-\frac{1}{2}\epsilon}^{\delta,M_{k}}(b,\eta)d\eta\right|+\\
 & C_{1}\tilde{\epsilon}<C_{3}'\tilde{\epsilon}\bigg|H_{k}\cap G_{1}\cap G_{2}\bigg)P(H_{k}\cap G_{1}\cap G_{2})\\
= & P\bigg(\left|\int_{\mathrm{D}(b,g,\tilde{\epsilon}-\frac{1}{2}\epsilon,C_{2}'\tilde{\epsilon},M_{k})}a^{-3/2}N(\psi_{a,b})da+\int_{\mathrm{D}(b,f,\tilde{\epsilon},\tilde{\epsilon},M_{k})\backslash\mathrm{D}(b,g,\tilde{\epsilon}-\frac{1}{2}\epsilon,C_{2}'\tilde{\epsilon},M_{k})}a^{-3/2}W_{f}(a,b)da\right|+\\
 & C_{1}\tilde{\epsilon}<C_{3}'\tilde{\epsilon}\bigg|H_{k}\cap G_{1}\cap G_{2}\bigg)P(H_{k}\cap G_{1}\cap G_{2})\\
\geq & P\left(\int_{1/M_{k}}^{M_{k}}a^{-3/2}\frac{\epsilon}{2}da+\int_{1/M_{k}}^{M_{k}}a^{-3/2}(M_{k}^{1/2}\epsilon+\tilde{\epsilon}+\frac{\epsilon}{2})da+C_{1}\tilde{\epsilon}<C_{3}'\tilde{\epsilon}\bigg|H_{k}\cap G_{1}\cap G_{2}\right)P(H_{k}\cap G_{1}\cap G_{2})\\
= & P\left(2(M_{k}^{1/2}-M_{k}^{-1/2})((M_{k}^{1/2}+1)\tilde{\epsilon}^{2}+1)+C_{1}<C_{3}'\bigg|H_{k}\cap G_{1}\cap G_{2}\right)P(H_{k}|G_{1}\cap G_{2})P(G_{1}\cap G_{2})\\
= & P(G_{1}\cap G_{2}),
\end{align*}
which completes the proof.

\end{proof}

Part (1) of Theorem \ref{SSStableThm2} is saying that the noise power gets spread out across the Synchrosqueezing time-frequency plane instead of
accumulating in a single component instantaneous frequency, despite the fact that the Synchrosqueezing frequencies are generally concentrated and are
not directly comparable to conventional Fourier frequencies (see \cite{Thakur2010}). Part (2) is the same statement for the entire component $A_{k}e^{2\pi i\phi_{k}}$,
including the amplitude. Note that the above argument can also be repeated for more general Gaussian processes such as ``$1/f$'' noise. In this case, the
covariances will change (to e.g. $\mathbb{E}(N(\psi_{a,b})\overline{N(\psi_{a,b})})=\epsilon^{2+p}\left\langle \psi,R\psi\right\rangle $),
but the pseudo-covariances will still be zero by the translation-invariance of the operator $R$, and the rest of the argument will be identical.


\section{\label{sec:analysis}Implementation Overview}

We now describe the Synchrosqueezing transform in a discretized form
that is suitable for efficient numerical implementation. We also discuss
several issues that arise in this process and how various parameters
are to be chosen in practice. We are given a vector $\tf\in\bbR^{n}$,
$n=2^{L+1}$, where $L$ is a nonnegative integer. Its elements, $\tf_{m},m=0,\ldots,n-1$,
correspond to a uniform discretization of $f(t)$ taken at the time
points $t_{m}=t_{0}+m\Delta t$. To prevent boundary effects, we pad
$\tf$ on both sides (using, e.g., reflecting boundary conditions). Figure \ref{fig:simple}
shows a graphical example of each step of the procedure outlined in this section.

\begin{figure}[h!]
  \centering
  \includegraphics{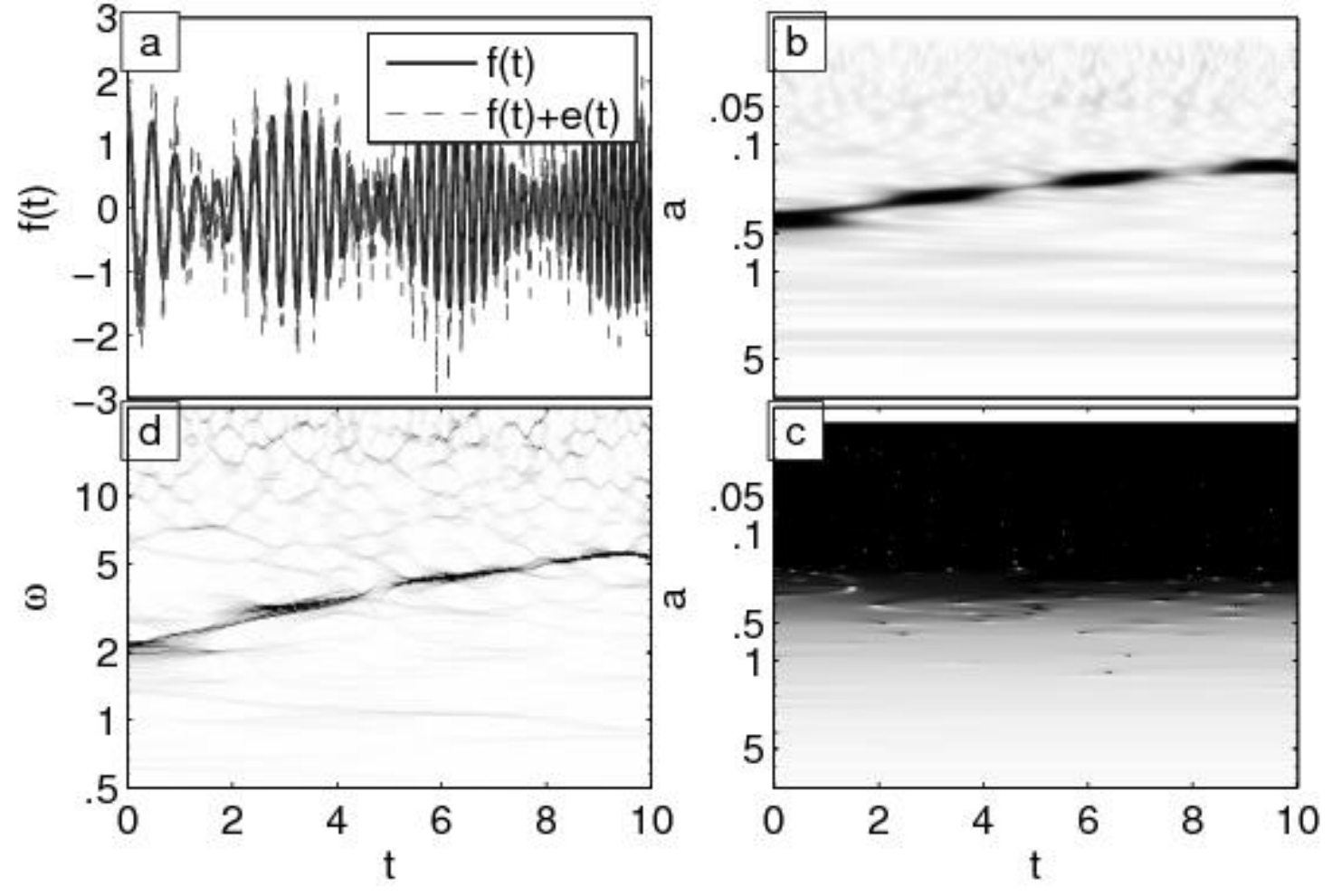}
  \caption{\label{fig:simple}
    Synchrosqueezing example for $f(t) = (1 + 0.6 \cos(2 t)) \cos(4
    \pi t+1.2 t^2)$ and additive noise $e(t) \sim \cN(0, 0.5^2)$.
    Panels in clockwise order:
    a) $f(t)$ and $f(t)+e(t)$ sampled, $n=1024$ points.
    b) CWT of $f$, $|W_f|$.
    c) Phase transform ${\omega}_{f}$.
    d) Synchrosqueezing transform $|T_f|$; with $\gamma=10^{-5}$ (see \S\ref{sec:gamma}).
  }
\end{figure}

\subsection{\label{sec:DWT}DWT of sampled signal: $\tW_{\tf}$}

We first choose an appropriate mother wavelet $\psi$. We pick $\psi$ such that
its Fourier transform $\wh{\psi}(\xi)$ (normalized as in Theorem
\ref{SSThm}) is concentrated in absolute value around some positive
frequency $\xi=\omega_{0}$, and is small and rapidly decaying elsewhere
(i.e. $\lim_{|t|\to\infty}P(t)\psi(t)=0$ for all polynomials $P$).
Many standard mother wavelets satisfy these properties, and we compare several
examples in \S\ref{sec:wcmp}.

The DWT samples the CWT $W_{f}$ at the locations $(a_{j},t_{m})$,
where $a_{j}=2^{j/n_{v}}\Delta t$, $j=1,\ldots,Ln_{v}$, and the
"voice number" $n_{v}$ \cite{Goupillaud1984} is a user-defined parameter
that affects the number of scales we work with (we have
found that $n_{v}=32$ or $64$ works well in practice). The DWT of $\tf$
can be calculated in $O(n_{v}n\log_{2}^{2}n)$ operations using the
FFT. We outline the steps below.

First note that $W_{f}(a,\cdot)=a^{-1/2}\overline{\psi(-\frac{\cdot}{a})}\star f$,
where $\star$ denotes the convolution. In the frequency domain, this
relationship becomes $\wh{W}_{f}(a,\xi)=a^{1/2}\wh{f}(\xi)\overline{\wh{\psi}(a\xi)}$.
We use this to calculate the DWT, $\tW_{\tf}(a_{j},t_{m})$. Let $\cF_{n}$
($\cF_{n}^{-1}$) be the standard (inverse) circular Discrete Fourier
Transform. Then 
\begin{equation}
\tW_{\tf}(a_{j},\cdot)=\cF_{n}^{-1}\left((\cF_{n}\tf)\odot\overline{\wh{\psi}_{j}}\right).\label{eq:Wxdisc}
\end{equation}
Here $\odot$ denotes elementwise multiplication and $\wh{\psi}_{j}$
is an $n$-length vector with $(\wh{\psi}_{j})_{m}=a_{j}^{1/2}\wh{\psi}(a_{j}\xi_{m})$;
$\xi_{m}$ are samples in the unit frequency interval: $\xi_{m}=2\pi m/n$,
$m=0,\ldots,n-1$.



\subsection{The phase transform: \textmd{\normalsize $\wt{\omega}_{\tf}$}}

The next step is to calculate the phase transform \eqref{Omega}. We first require a slight modification of the definition
\eqref{Omega},
\begin{equation}
\omega_{f}(a,b)=\frac{1}{2\pi}\Im{(W_{f}(a,b))^{-1}\partial_{b}W_{f}(a,b)}.\label{eq:omegax}
\end{equation}
In theory Eqs. \eqref{eq:omegax} and \eqref{Omega} are equivalent, and in
practice \eqref{eq:omegax} is a convenient way to obtain a real-valued frequency
from \eqref{Omega}.  We denote the discretization of $\omega_{f}$
by $\wt{\omega}_{\tf}$.

In practice, signals have noise and other artifacts due to, e.g.,
sampling errors, and computing the phase of $W_{f}$ is unstable when
$\abs{W_{f}}\approx0$. Therefore, we choose a hard threshold parameter
$\gamma>0$ and disregard any points where $\abs{W_{f}}\leq\gamma$
. The exact choice of $\gamma$ is discussed in Sec. \ref{sec:gamma}.
We use this to define the numerical support of $\tW_{\tf}$, on which
$\omega_{f}$ can be estimated:

\begin{center}
$\wt{\cS}_{\tf}^{\gamma}(m)=\set{j:\abs{\tW_{\tf}(a_{j},t_{m})}>\gamma}$,
for $m=0,\ldots,n-1$. 
\par\end{center}

The derivative in \eqref{eq:omegax} can be calculated by taking finite
differences of $\tW_{\tf}$ with respect to $m$, but Fourier transforms
provide a more accurate alternative. Using the property $\wh{\partial_{b}W_{f}}(a,\xi)=2\pi i\xi\wh{W_{f}}(a,\xi)$,
we estimate the phase transform, for $j\in\wt{\cS}_{\tf}^{\gamma}(m)$,
as 

\begin{center}
$\wt{\omega}_{\tf}(a_{j},t_{m})=\frac{1}{2\pi}\Im{\left(\tW_{\tf}(a_{j},t_{m})\right)^{-1}\partial_{b}\tW_{\tf}(a_{j},t_{m})},$ 
\par\end{center}

with the derivative of $W_{f}$ estimated via (e.g., \cite{Tadmor1986}) 

\begin{center}
$\partial_{b}\tW_{\tf}(a_{j},\cdot)=\cF_{n}^{-1}\left((\cF_{n}\tf)\odot\wh{\partial\psi}_{j}\right),$ 
\par\end{center}

where $(\wh{\partial\psi}_{j})_{m}=2\pi ia_{j}^{1/2}\xi_{m}\wh{\psi}(a_{j}\xi_{m})/\Delta t$
for $m=0,\ldots,n-1$.

The normalization of $\wt{\omega}$ corresponds to a dominant,
constant frequency of $\alpha$ when $f(t)=\cos(2\pi\alpha t)$. This
allows us to transition from the time-scale plane to a time-frequency
plane according to the reassignment map $(a,b)\to(\omega(a,b),b)$. Note
that the phase transform is not the instantaneous frequency itself except
in some simple cases, but contains requisite ``frequency information'' that
we use to recover the actual frequencies in the next step.

\subsection{Synchrosqueezing in the time-frequency plane: ${T_{f}(\omega,b)}$}

We now compute the Synchrosqueezing transform using the reassigned
time-frequency plane. Suppose we have some ``frequency divisions''
$\set{w_{l}}_{l=0}^{\infty}$ with $w_{0}>0$ and $w_{l+1}>w_{l}$
for all $l$. Let the frequency bin $\cW_{l}$ be given by $\{w'\in\bbR:|w'-w_{l}|<|w'-w_{l'}|\,\forall l'\not=l\}$,
or in other words, the set of points closer to $w_{l}$ than any other
$w_{l'}$. We define the discrete-frequency Wavelet Synchrosqueezing
transform of $f$ by 
\begin{equation}
T_{f}(w_{l},b)=\int_{\set{a:\omega_{f}(a,b)\in\cW_{l},|W_{f}(a,b)|>\gamma}}W_{f}(a,b)a^{-3/2}da.\label{eq:Tfdef}
\end{equation}
This is essentially the limiting case of the definition \eqref{SS}
as $\delta\to0$ (note the argument in \eqref{ApproxId} and see also
\cite[p. 5-6]{Daubechies2010}), but with the frequency variable $\eta\in\bbR$
replaced by the discrete intervals $\cW_{l}$. Note that the discretization
$\tW_{\tf}$ is given with respect to $n_{a}=Ln_{v}$ log-scale samples
of the scale $a$, so we correspondingly discretize \eqref{eq:Tfdef}
over a logarithmic scale in $a$. The transformation $a(z)=2^{z/n_{v}}$,
$da(z)=a\frac{\log2}{n_{v}}dz$, leads to the modified integrand ${W_{f}(a,b)a^{-1/2}\frac{\log2}{n_{v}}dz}$
in \eqref{eq:Tfdef}.

To choose the frequency divisions $w_{l}$, note that the time step
$\Delta t$ limits the range of frequencies that can be estimated.
One form of the Nyquist sampling theorem shows that the maximum frequency is $\ol{w}=w_{n_{a}-1}=\frac{1}{2\Delta t}$.
Since $f$ is discretized over an interval of length $n\Delta t$,
the fundamental (minimum) frequency is $\ul{w}=w_{0}=\frac{1}{n\Delta t}$.
Combining these bounds on a logarithmic scale, we get the divisions
$w_{l}=2^{l\Delta w}\ul{w}$, $l=0,\ldots,n_{a}-1$, where $\Delta w=\frac{1}{n_{a}-1}\log_{2}(n/2)$.

We can now calculate a fully discretized estimate of \eqref{eq:Tfdef},
denoted by $\tT_{\tf}$. Since we have already tabulated $\wt{\omega}_{\tf}$
and $\wt{\omega}_{\tf}(a_{j},t_{m})$ lands in at most one frequency
bin $\cW_{l}$, the integral in \eqref{eq:Tfdef} can be computed
efficiently by finding the associated $\cW_{l}$ for each $(a_{j},t_{m})$
and adding it to the appropriate sum. This results in $O(n_{a}n)$
computations for the entire Synchrosqueezed plane $\tT_{\tf}$. We
summarize this approach in pseudocode in Alg. \ref{alg:Tf}.

\begin{algorithm}[h!]
\caption{Fast calculation of $\tT_\tf$ for fixed $m$}
\label{alg:Tf}
  \begin{algorithmic}
    \small
    \FOR[Initialize $\tT$ for this $m$]{$l = 0$ to $n_a-1$}
    \STATE $\tT_{\tf}(w_l, t_m) \leftarrow 0$
    \ENDFOR
    \FORALL[Calculate \eqref{eq:Tfdef}]{$j \in
      \wt{\cS}^\gamma_{\tf}(m)$}
    \STATE \COMMENT{Find frequency bin via $w_l = 2^{l \Delta w}
      \ul{w}$, and $\wt{\omega}_{\tf} \in \cW_l$}
    \STATE $l \leftarrow
    \min\left( \max\left( \textrm{ROUND} \left[ \frac{1}{\Delta w} \log_2 \left(
      \frac{\wt{\omega}_{\tf}(a_j,b_m)}{\ul{w}} \right) \right], 0 \right), n_a-1 \right) $
    \STATE \COMMENT{Add normalized term to appropriate integral;
      $\Delta z = 1$}
    \STATE $\tT_{\tf}(w_l,t_m) \leftarrow \tT_{\tf}(w_l,t_m) + \frac{\log 2}{n_v}
    \tW_{\tf}(a_j,t_m) a^{-1/2}_j $
    \ENDFOR
  \end{algorithmic}
\end{algorithm}

We remark that as an alternative, the frequency divisions $w_l$ can be spaced linearly, instead of the
logarithmic scale we use in keeping with the discretization of the CWT in Section \ref{sec:DWT}. Examples of
this approach can be found in \cite{Wu2012}, but in practice we find that there are no significant differences
either way. In principle, the CWT itself can be discretized linearly as well, but the approach we took in
Section \ref{sec:DWT} is standard and is preferred for its computational efficiency (see \cite{Da92,Wavelab}).

\subsection{Component reconstruction}

We can finally recover each component $f_{k}$ from $\tT_{\tf}$ by inverting the CWT (integrating) over the
frequencies $w_{l}$ that correspond to the $k$th component, an approach similar to filtering on a conventional TF plot.
Let $l\in\mathcal{L}_{k}(t_{m})$ be the indices of a small frequency band around the curve of $k$th component in
the phase transform space (based on the results of Thm. \ref{SSStableThm} and Thm. \ref{SSStableThm2} parts 2).
These frequencies can be selected by hand or estimated via a standard least-squares ridge extraction method \cite{CHT97},
as done in the Synchrosqueezing Toolbox. Then, using the fact that $f_{k}$ is real, we have
\begin{equation}
f_{k}(t_{m})=2\cR_{\psi}^{-1}\Re{\sum_{l\in\mathcal{L}_{k}(t_{m})}\tT_{\tf}(w_{l},t_{m})},\label{eq:Tfrecon}
\end{equation}

where $\cR_{\psi}$ is the normalization constant from Theorem \eqref{SSThm}. \\

\subsection{Selecting the threshold $\gamma$\label{sec:gamma}}

The hard wavelet threshold $\gamma$ effectively decides the lowest
CWT magnitude at which $\omega$ is deemed trustworthy. In an
ideal setting wherein the signal is not corrupted by noise, this threshold
can be set based on the machine epsilon (we suggest $10^{-8}$ for
double precision floating point systems). In practice, $\gamma$ can
be seen as a hard threshold on the wavelet representation (shrinking
small magnitude coefficients to $0$), and its value determines the
level of filtering.

In \cite{Donoho1994}, a nearly minimax optimal procedure was proposed
for denoising sufficiently smooth signals corrupted by additive white
noise. This algorithm consists of soft- or hard-thresholding the wavelet
coefficients of the corrupted signal, followed by inversion of the
filtered wavelet representation. In \cite{Donoho1995}, this estimator
was also shown to be nearly optimal in terms of root mean square error.
The asymptotically optimal threshold is $\sqrt{2\log n}\cdot\sigma$,
where $n$ is the signal length and $\sigma^{2}$ is the noise power.
Following \cite{Donoho1994}, the noise power can be estimated from
the Median Absolute Deviation (MAD) of the finest level wavelet coefficients.
This is the threshold we suggest and use throughout our simulations:
\[
\gamma=1.4826\sqrt{2\log n}\cdot\text{MAD}(|\tW_{\tf}|_{1:n_{v}})
\]
where $1.4826$ is the multiplicative factor relating the MAD of a
Gaussian distribution to its standard deviation, and $|\tW_{\tf}|_{1:n_{v}}$
are the wavelet coefficient magnitudes at the $n_{v}$ finest scales
(the first octave).


\section{\label{sec:wmisc}Numerical Simulations}

In this section, we provide several numerical examples that illustrate the
ideas in Sec. \ref{sec:main} and \ref{sec:analysis} and show how
Synchrosqueezing compares to a variety of other time-frequency transforms
in current use. The MATLAB scripts used to generate the figures for these
examples are available at \cite{SSToolbox}.

\subsection{Comparison of Synchrosqueezing with the CWT, STFT and EEMD}
We first compare the Synchrosqueezing time-frequency decomposition to the
continuous wavelet transform (CWT) and the short-time Fourier transform (STFT)
\cite{Oppenheim1999}. We show its superior precision, in both time
and frequency, at identifying the components of complicated oscillatory
signals.  We then show its ability to reconstruct (via filtering)
an individual component from a curve in the time-frequency plane. We also compare
the recovered component with the results of the ensemble empirical mode
decomposition (EEMD) method (see \cite{Wu2005} for details).

\begin{figure}[h]
  \centering
  \includegraphics[width=.8\columnwidth]{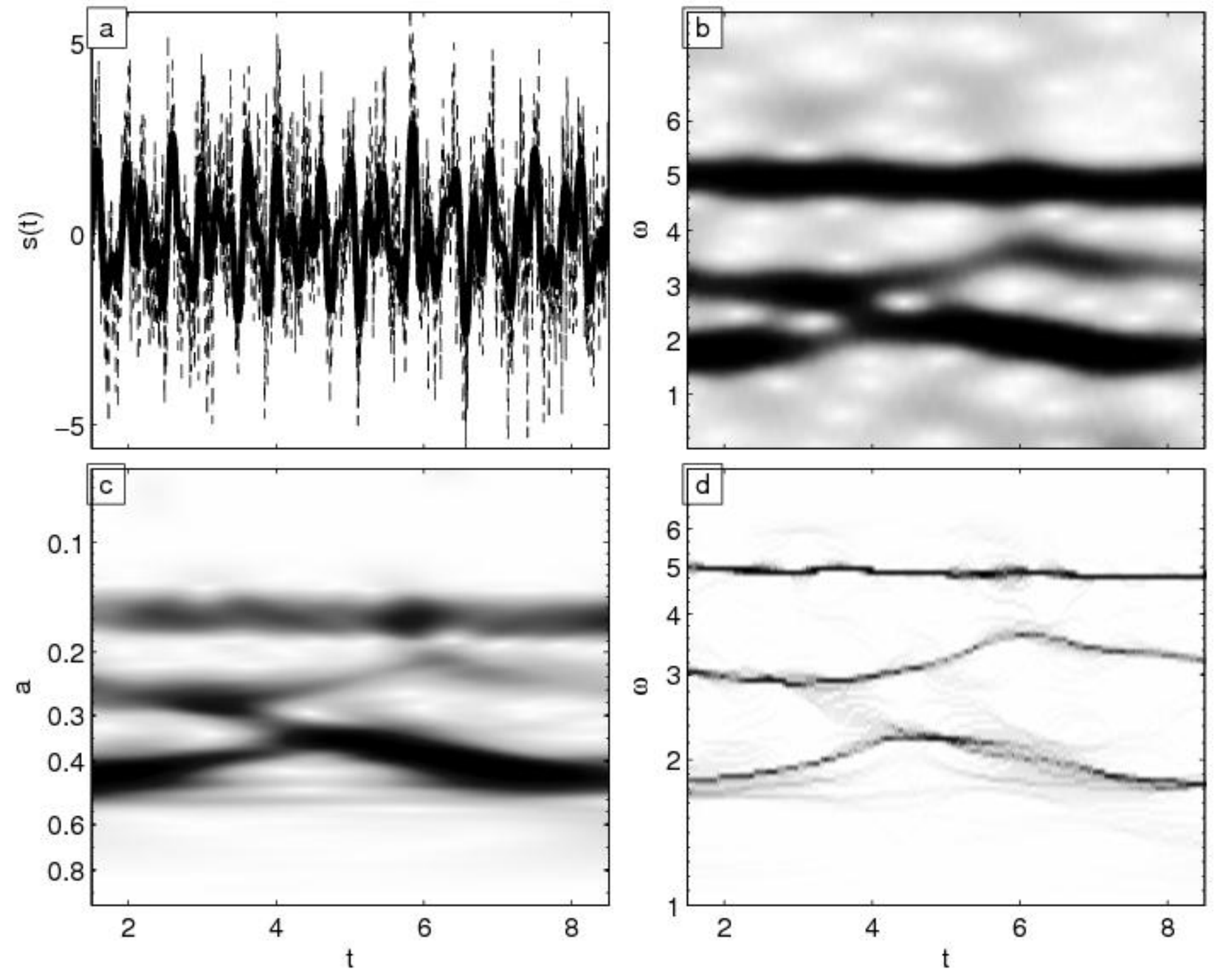}
  \caption{\footnotesize\label{fig:cmpstftwave} 
    Comparison of Synchrosqueezing with the STFT and CWT.  (a) Synthetic signal $s(t)$ (bold), corrupted with
 	 noise (dashed), shown for  $t \in [2,8]$.  (b) STFT of signal $s(t)$. (c) CWT of signal $s(t)$. (d) Synchrosqueezing plot $T_s(\omega,t)$.
  }
\end{figure}

In Fig. \ref{fig:cmpstftwave} we consider a signal $s(t)=s_1(t)+s_2(t)+s_3(t)+N(t)$ defined on
$t \in [0,10]$ that contains different kinds of time-varying AM and FM modulation.  It is composed of the following components:
\begin{align*}
 s_1(t) &= (1+0.2\cos(t)) \cos(2\pi(2t + 0.3\cos(t))), \\
 s_2(t) &= (1+0.3\cos(2t)) e^{-t/15} \cos(2\pi(2.4t + 0.5t^{1.2} + 0.3\sin(t))) \\
 s_3(t) &= \cos(2\pi(5.3t + 0.2t^{1.3})).
\end{align*}
The signal is discretized to $n=2048$ points and corrupted by additive Gaussian white noise $N(t)$ with noise power $\sigma^2 = 2.4$,
leading to an SNR of $-2.6$~dB.

To make the comparison consistent (as the $\gamma$ threshold in Synchrosqueezing has a denoising effect), we first denoise
the signal using the Wavelet hard-thresholding methodology of~\S\ref{sec:gamma}. We then feed this denoised signal
to the STFT, CWT, and Synchrosqueezing transforms. We use the shifted bump wavelet (see~\S\ref{sec:wcmp}) and $n_v = 32$
for both the CWT and Synchrosqueezing transforms, and a Hamming window with length 400 and overlap of length 399 for the
STFT. These STFT parameters are selected to have a representation visually balanced between time and frequency
resolution~\cite{Oppenheim1999}.\\

The component $s_3$ is close to a Fourier harmonic and is clearly identified in the Synchrosqueezing plot $T_s$
(Fig. \ref{fig:cmpstftwave}(d)) and the STFT plot (Fig. \ref{fig:cmpstftwave}(b)), though the frequency
estimate is more precise in $T_s$. The other two components have time-varying instantaneous frequencies and
can be clearly distinguished in the Synchrosqueezing plot, while there is much more smearing and distortion in them
in the STFT and CWT. The temporal resolution of the CWT and STFT is also significantly lower than for Synchrosqueezing
due to the selected parameters. A shorter time window or wavelet will provide higher temporal resolution, but lower
frequency resolution and more smearing between the three components.

\begin{figure}[h]
  \centering
  \includegraphics[width=1\columnwidth]{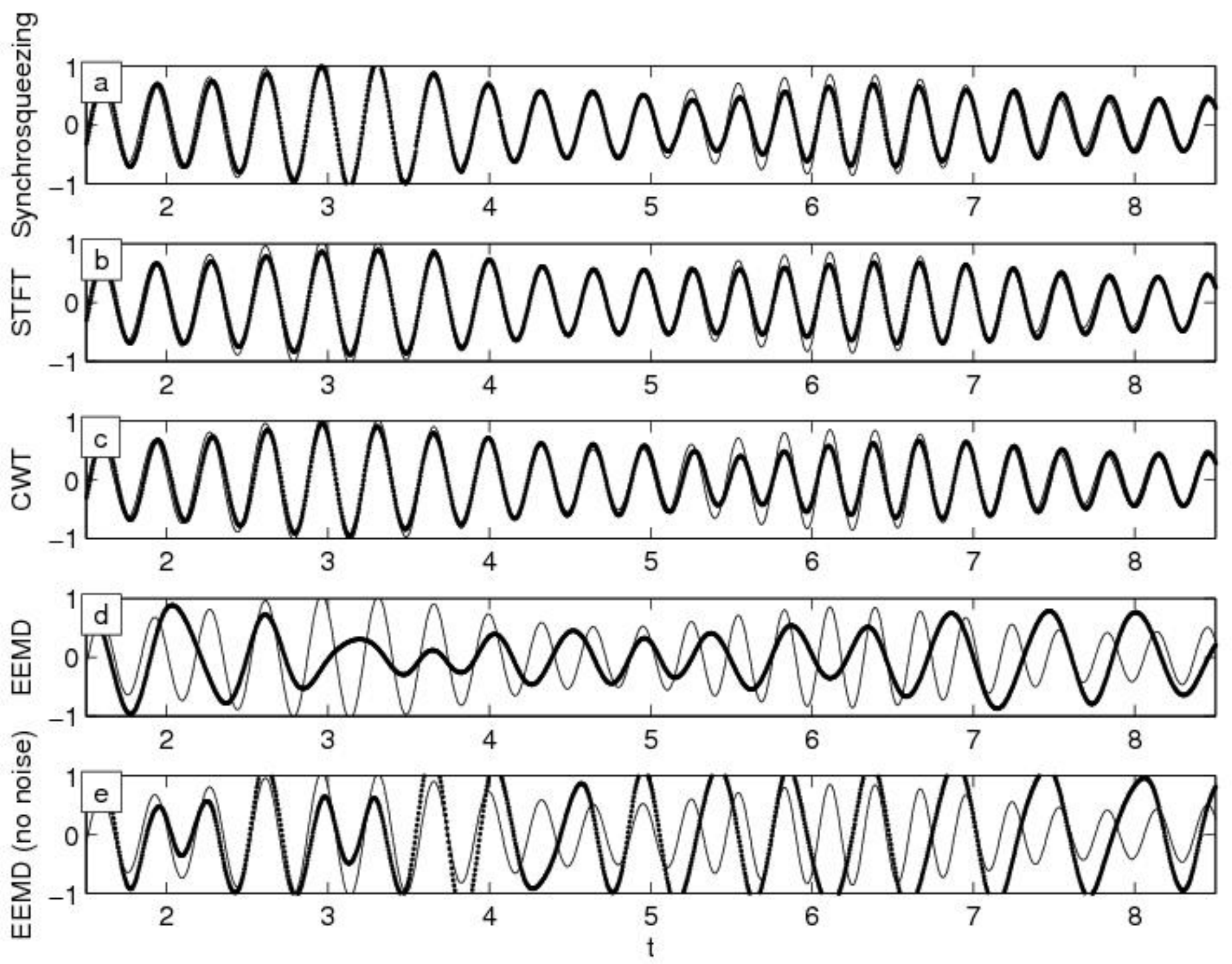}
  \caption{\footnotesize\label{fig:cmpstftwave2} 
    (a-c) Reconstruction of the component $s_{2}$ on $[2,8]$ performed by inverting Synchrosqueezing (a), CWT (b) and STFT (c), shown as dotted curves.
	(d-e) The EEMD extraction of $s_{2}$ with $50$ ensembles performed on the signal $s$ (d), and on the same signal without any noise and one ensemble
	(e). The original component $s_{2}$ is shown in solid curves for reference.}
\end{figure}

Fig. \ref{fig:cmpstftwave2} shows the component $s_{2}$ reconstructed from the TF plots in Fig. \ref{fig:cmpstftwave} by
inverting each transform in a small band around the curve of $s_{2}$. All three time-frequency methods provide comparable results
and pick up the component reasonably accurately, although the AM behavior around $t\in[5,7]$ is slightly smothered out as a
result of the noise. On the other hand, EEMD exhibits a poor amplitude recovery and a drifting phase over time, even when applied
to the original signal without any noise. In general, EMD/EEMD is sensitive to amplitude changes over time that impose strong
requirements on the frequency separation between the components \cite{Flandrin2008}, while Synchrosqueezing and the other
time-frequency methods produce good results as long as the bandwidth of the mother wavelet or window is small enough,
according to Theorem \ref{SSThm}.

\subsection{Comparison of Synchrosqueezing with Reassignment Techniques}

We next compare the analysis part of Synchrosqueezing to two of the most common time-frequency reassignment (TFR) methods,
based on the spectrogram and the Wigner-Ville distribution (see \cite{Flandrin1999}, ch. 4 for details). We apply these
techniques to $s(t)$, the signal from the last example, for $t \in [2,8]$ and with the noise increased to
$\sigma^2 = 5$ ($-5.8$~dB SNR). The results are shown in Fig. \ref{fig:cmptfr}.

\begin{figure}[h]
  \centering
  \includegraphics{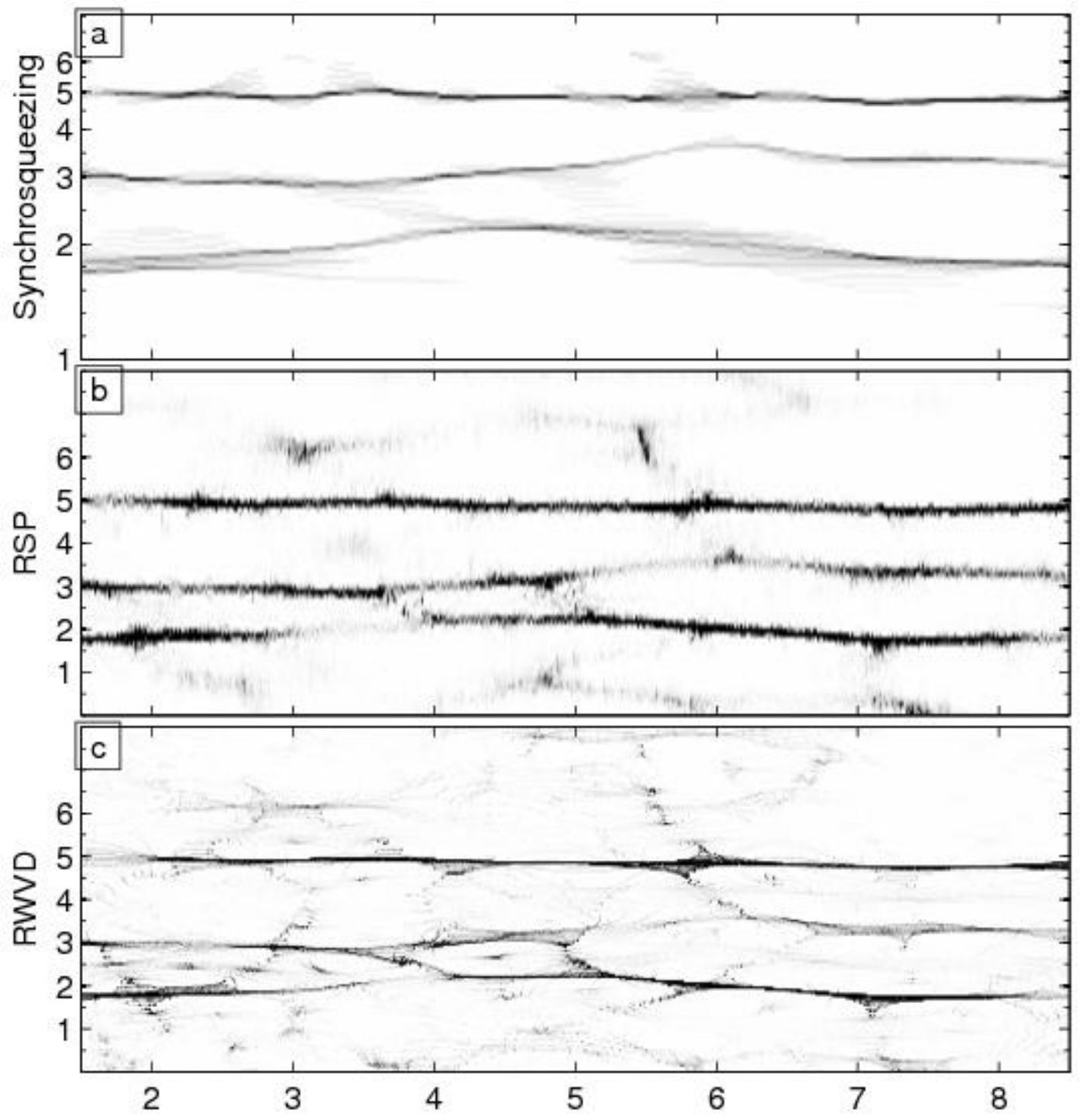}
   \caption{\label{fig:cmptfr}
     (a) Synchrosqueezing $\wt{T}_{\tf}$ of $\tf$.
     (b) Reassigned spectrogram / STFT of $\tf$ (RSP).
     (c) Reassigned smoothed pseudo-WVD of $\tf$ (RWVD).
   }
\end{figure}

Synchrosqueezing can be understood as a variant of the standard TFR methods. In TFR methods,
the directional reassignment vector is computed in both time and frequency from the magnitude
of the STFT or WVD, which is then used to remap the energies in the TF plane of a signal.
In contrast, the Synchrosqueezing transform can be thought of as a reassignment vector only in
the frequency direction. The fact that there are no time shifts in the TF plane is what allows
the reconstruction of the signal to be possible. We note that in Fig. \ref{fig:cmptfr}, the
Synchrosqueezing TF plot contains fewer spurious components than the other TFR plots. The other TFR methods
exhibit additional clutter in the TF plane caused by the noise, and the reassigned
WVD also contains traces of an extra curve between the second and third components, a result of
the quadratic cross-terms that are characteristic of the WVD \cite{Flandrin1999}.

\subsection{Nonuniform Samples and Spline Fitting}

We now demonstrate how Synchrosqueezing analysis and extraction works for a
three-component signal that has been irregularly sampled. For $t \in [2,8]$, let
\begin{align}
f(t) &= (1+0.5 \cos(t)) \cos(4\pi t)\\				\nonumber
     &+ 2e^{-0.1t}\cos(2\pi(3t+0.25\sin(1.4t)))\\	\nonumber
     &+ (1+0.5\cos(2.5t)) \cos(2\pi (5t+2t^{1.3})), \nonumber
\end{align}
and let the sampling times be perturbations of uniformly spaced times having the form
$t'_{m}=\Delta t_1 m + \Delta t_2 u_{m}$, where $\{u_{m}\}$ is sampled from the uniform distribution
on $[0,1]$. We take $\Delta t_1=11/300$ and $\Delta t_2=11/310$, which leads to approximately $165$
samples on the interval $[2,8]$ and an average sampling rate of $27.2$, or about three times the
maximum instantaneous frequency of $9.85$. As indicated in Cor. \ref{cor:splinestable}, we account for the
nonuniform sample spacing by fitting a cubic spline through $(t'_m,f(t'_m))$ to get the interpolant
$f_s(t)$, discretized on the finer grid $t_m = m \Delta t$ with $\Delta t=10/1024$ and $m=0,\ldots,1023$.
The resulting vector, $\tf_s$, is a discretization of the original signal plus a spline error term $e(t)$.\\

\begin{figure}[h]
  \centering
  \includegraphics[width=.7\columnwidth]{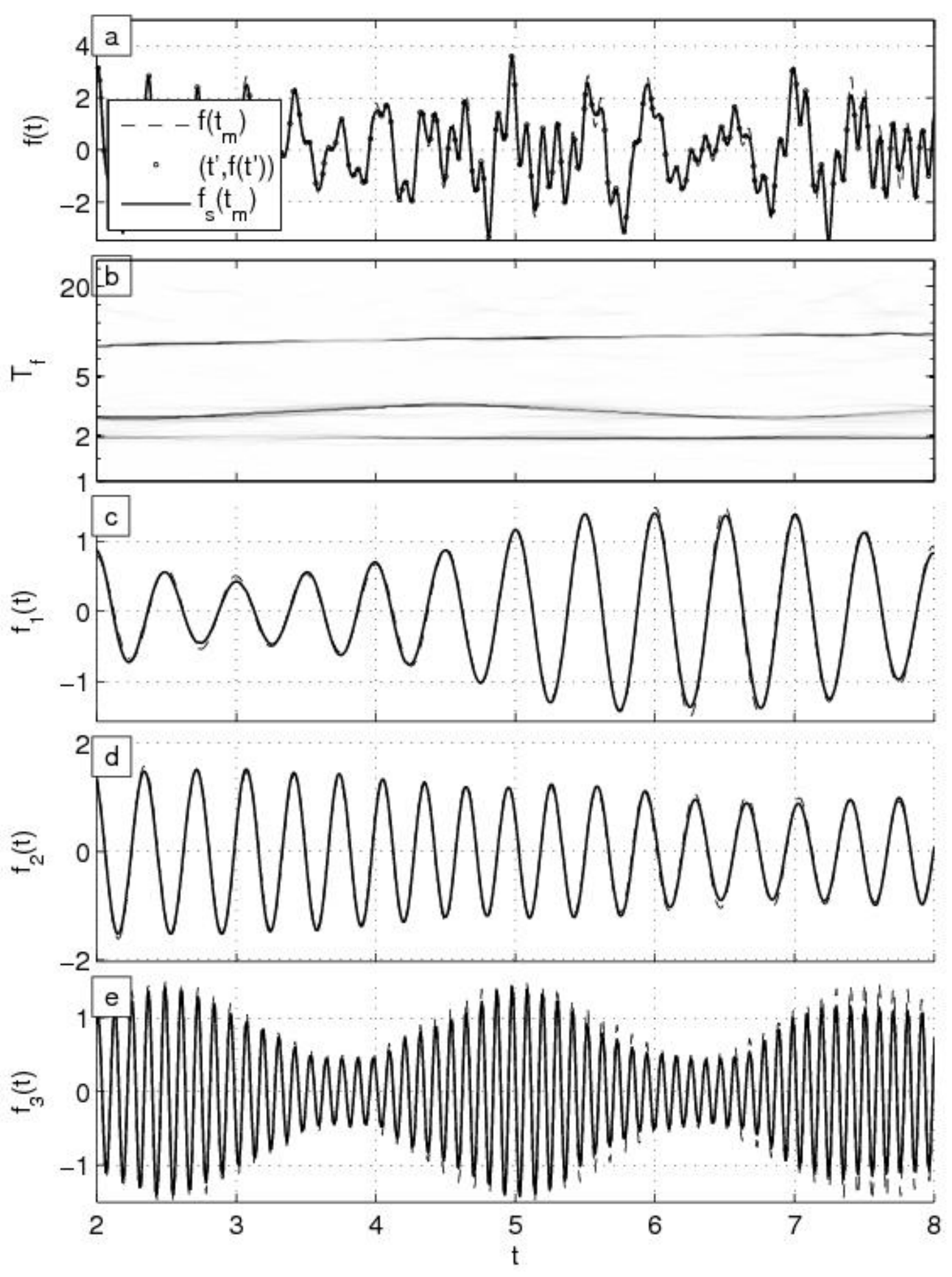}
   \caption{\label{fig:nonunif}
     (a) Nonuniform samples of $f$, with spline interpolant $\tf_s$ (solid), and original signal $f$ (dashed).
     (b) Synchrosqueezing TF plane $\wt{T}_{\tf_s}$.
     (c-e) Extracted components $\tf^*_k$ for $k=1,2,3$ (solid) compared to originals $\tf_k$ (dashed).
   }
\end{figure}

Fig. \ref{fig:nonunif}(a-e) shows the Synchrosqueezing TF plot $\tf_s$ and the three reconstructed components.
The spline interpolant approximates the original signal closely, except for a few oscillations for $t>7.3$ where
the highest frequencies of $f$ occur. The Synchrosqueezing results are largely unaffected by the errors and have no
spurious spectral information in the TF plot. The effect of the interpolation errors for $t>7.3$ is also localized
in time and only influences the AM recovery of the third component, which contains the highest frequencies and is
the most difficult to recover as indicated by Thm. \ref{SSStableThm}. In general, however, we find that components
close to the Nyquist frequency are picked up fairly accurately as long as the mother wavelet is chosen according
to Theorem \ref{SSThm} and the components are spaced sufficiently far apart (for cases where multiple high frequency
components are close together, see the STFT Synchrosqueezing approach in \cite{Thakur2010}).

\subsection{\label{sec:wcmp}Invariance to the underlying transform}

As a final example, we show the effect of the underlying mother wavelet on the Synchrosqueezing transform.
As discussed in \cite{Daubechies2010}, Synchrosqueezing is largely invariant to the choice of the mother
wavelet, and the main differences one sees in practice are due to the wavelet's relative concentrations in
time and frequency (in particular, how far away its frequency content is from zero), as opposed to its precise shape.

Fig. \ref{fig:wcmp} shows the effect of Synchrosqueezing on the discretized spline signal $\tf_s$ from the
last example, using three different complex CWT mother wavelets. These wavelets are:
\begin{align*}
&\textbf{a. Morlet (shifted Gaussian)} \\
&\qquad \wh{\psi}_a(\xi) \propto \exp(-2\pi^2(\mu-\xi)^2),
  \quad \xi \in \bbR\\
&\textbf{b. Complex Mexican Hat} \\
&\qquad \wh{\psi}_b(\xi) \propto \xi^2 \exp(-2\pi^2\sigma^2 \xi^2),
  \quad \xi > 0\\
&\textbf{c. Shifted Bump} \\
&\qquad \wh{\psi}_d(\xi) \propto \exp\left(- (1- ((2\pi\xi-\mu)/\sigma )^2 )^{-1} \right), \\
&\qquad \xi \in [\sigma(\mu-1), \sigma(\mu+1)]
\end{align*}
where for $\psi_a$ we use $\mu=1$, for $\psi_b$ we use $\sigma=1$, and for $\psi_c$ we use $\mu=5$ and $\sigma=1$.
These respectively correspond to about $\Delta=0.5$, $0.25$ and $0.16$ in Thm. \ref{SSThm}. We find that, as
indicated by Thm. \ref{SSThm}, the most accurate representation is given by the bump wavelet $\psi_c$, whose
frequency support is the smallest and exactly (instead of approximately) positive and finite.

\begin{figure}[h]
  \centering
  \includegraphics[width=.8\columnwidth]{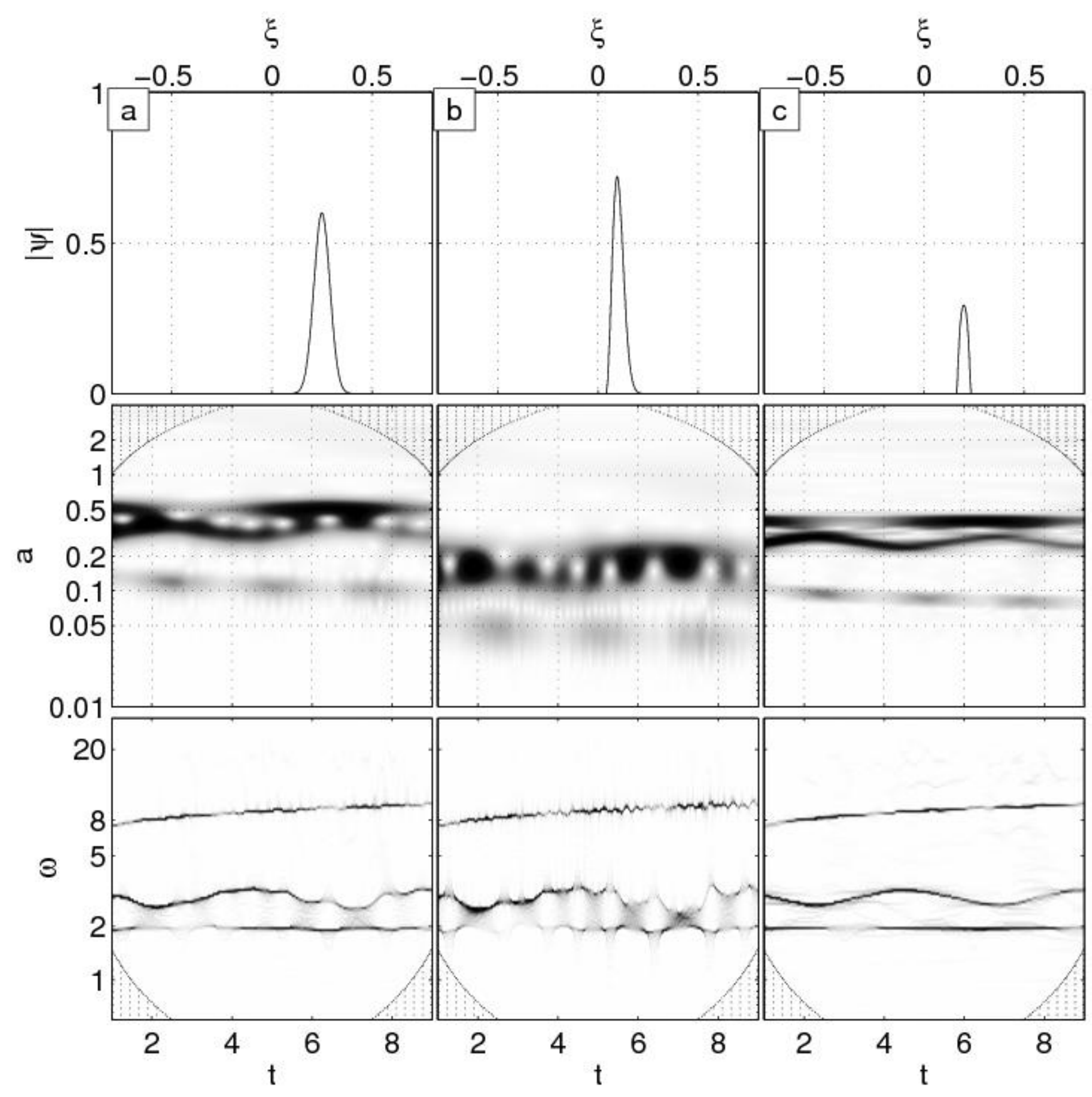}
  \caption{\footnotesize\label{fig:wcmp} 
    Wavelet and Synchrosqueezing transforms of $\tf_s$.
    Columns (a-c) represent choice of mother wavelet $\psi_a \ldots \psi_c$.
    Top row: $|2 \wh{\psi}(4\xi)|$.  Center row: $|W_{f_s}|$.  Bottom row:
    $|T_{f_s}|$.
    }
\end{figure}



\section{\label{sec:SSpaleo}Aspects of the mid-Pleistocene transition}

In this section, we apply Synchrosqueezing to analyze the characteristics
of a calculated index of the incoming solar radiation (insolation) and
paleoclimate records of repeated transitions between glacial (cold)
and interglacial (warm) climates, i.e., ice age cycles, primarily
during the Pleistocene epoch (from $\approx$\,1.8\,Myr to
$\approx$\,12\,kyr before the present). The analysis of time series is
crucial for paleoclimate research becuase its empirical base consists
of a growing collection of long deposited records. \\

The Earth's climate is a complex, multi-component, nonlinear system
with significant stochastic elements \cite{Pierrehumbert2010}. The key
external forcing field is the insolation at the top of the
atmosphere (TOA). Local insolation has predominantly harmonic
characteristics in time (diurnal cycle, annual cycle and
very long Milankovi\'{c} orbital cycles) enriched by the solar variability.
The response of the planetary climate, which varies at all time scales \cite{Huybers2006},
also depends on random perturbations (e.g., volcanism), nonstationary solid boundary
conditions (e.g., plate tectonics and global ice distribution),
internal variability and feedback (e.g., global carbon
cycle). Various paleoclimate records, called proxies, provide us with
information about past climates beyond observational records. These proxies
are biogeochemical tracers, i.e. molecular or isotopic properties,
imprinted into various types of deposits (e.g., deep-sea sediment, ice
cores, etc.), and they indirectly represent physical conditions (e.g. temperature)
at the time of deposition. We focus on climate variability during the
last 2.5\,Myr (that also includes the late Pliocene and the Holocene) as recorded by
${\delta}^{18}O$ The oxygen isotope variations in seawater are expressed
as deviations of the ratio of $^{18}O$ to $^{16}O$ with respect to the
present-day standard. Carbonate shells of foraminifera plankton (benthic forams)
at the bottom of the ocean record ${\delta}^{18}O$ changes in seawater during
their growth. The benthic ${\delta}^{18}O$ indicates changes in the global sea
level, ice volume and deep ocean temperature. During the buildup of land ice sheets
and the decrease in sea level in cold climates, the lighter $^{16}O$ evaporates more
readily than $^{18}O$ and accumulates in ice sheets, leaving the surface water
enriched with $^{18}O$. At the same time, the inclusion of $^{18}O$ during the
formation of carbonate shells records the ambient seawater temperature of the
benthic forams \cite{Lea2003}.  \\


\begin{figure}[h!]
   \centering
   \includegraphics[width=.7\columnwidth]{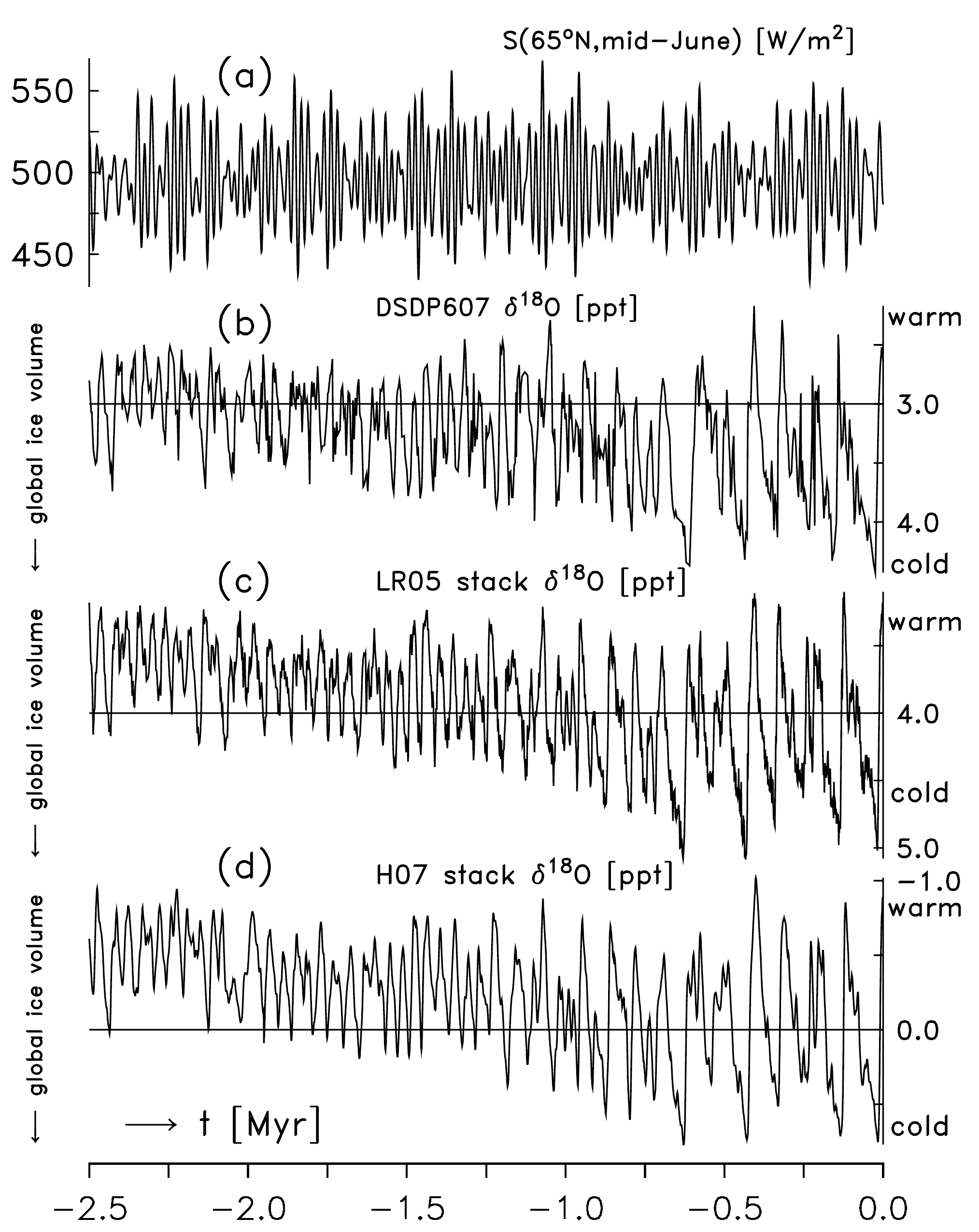}
   \caption{\label{fig:paleo_ts} (a) Calculated June 21 TOA insolation
     flux at $65^{o}N$: $f_{SF}$.  Climate response as recorded by
     benthic forams ${\delta}^{18}O$ (b) in the DSDP607 core $f_{CR1}$, (c) in
     the LR05 stack $f_{CR2}$, and (d) in the H07 stack $f_{CR3}$.}
\end{figure}

We first examine a calculated element of the daily TOA solar forcing
field. Fig. \ref{fig:paleo_ts}(a) shows $f_{SF}$, the mid-June
insolation at $65^{o}N$ at 1\,kyr intervals \cite{Berger1991}. This
TOA forcing index has been widely used to gain
insight into the timing of advances and retreats of ice sheets in the
Northern Hemisphere during this period, based on the classic
Milankovi\'{c} hypothesis that summer solstice insolation at
$65^{o}N$ paces ice age cycles (e.g., \cite{Hays1976, Berger1988}). The CWT and
Synchrosqueezing spectral decompositions (using the shifted bump mother
wavelet as in the rest of the paper), in
Fig. \ref{fig:paleo_synsq}(a) and Fig. \ref{fig:paleo_synsq}(e)
respectively, show the key time-varying oscillatory components of
$f_{SF}$. Both panels confirm the presence of strong
precession cycles (at periodicities $\tau$=19\,kyr and 23\,kyr),
obliquity cycles (primary at 41\,kyr and secondary at 54\,kyr), and
very weak eccentricity cycles (primary periodicities at 95\,kyr and
124\,kyr, and secondary at 400\,kyr). However, the Synchrosqueezing
spectral structure is far more concentrated along the
frequency (periodicity) direction than the CWT. \\


We next analyze the North Altantic and global climate response during
the last 2.5\,Myr as deposited in benthic ${\delta}^{18}O$ in long
sediments cores (in which deeper layers contain forams settled further back in time).
Fig.  \ref{fig:paleo_ts}(b) shows
$f_{CR1}$: benthic ${\delta}^{18}O$, sampled at irregular time intervals
from a single core, DSDP Site 607, in the North Atlantic
\cite{Ruddiman1989}.  Fig. \ref{fig:paleo_ts}(c)
shows $f_{CR2}$: the orbitally tuned benthic ${\delta}^{18}O$ stack of
\cite{Lisiecki2005} (LR05). It is an average of 57 globally
distributed records placed on a common age
model using a graphic correlation technique
\cite{Lisiecki2002}. Fig. \ref{fig:paleo_ts}(d) shows $f_{CR3}$: the benthic ${\delta}^{18}O$ stack of \cite{Huybers2007} (H07) calculated from
14 cores (mostly in the Northern Hemisphere) using the extended
depth-derived age model free from orbital tuning \cite{Huybers2004}. The ${\delta}^{18}O$ records included in these stacks vary
over different ranges primarily due to different ambient temperatures
at different depths of the ocean floor at core drill sites. Also, prior to combining the cores in
the H07 stack, the record mean between 0.7\,Myr ago and the present
was subtracted from each ${\delta}^{18}O$ record, so we have
different vertical ranges in Fig. \ref{fig:paleo_ts}(b) through Fig. \ref{fig:paleo_ts}(d). All ${\delta}^{18}O$ records are
spline interpolated to 1\,kyr intervals prior to the spectral analysis.\\

\begin{figure}[h!]
   \centering
   \includegraphics[width=.95\columnwidth]{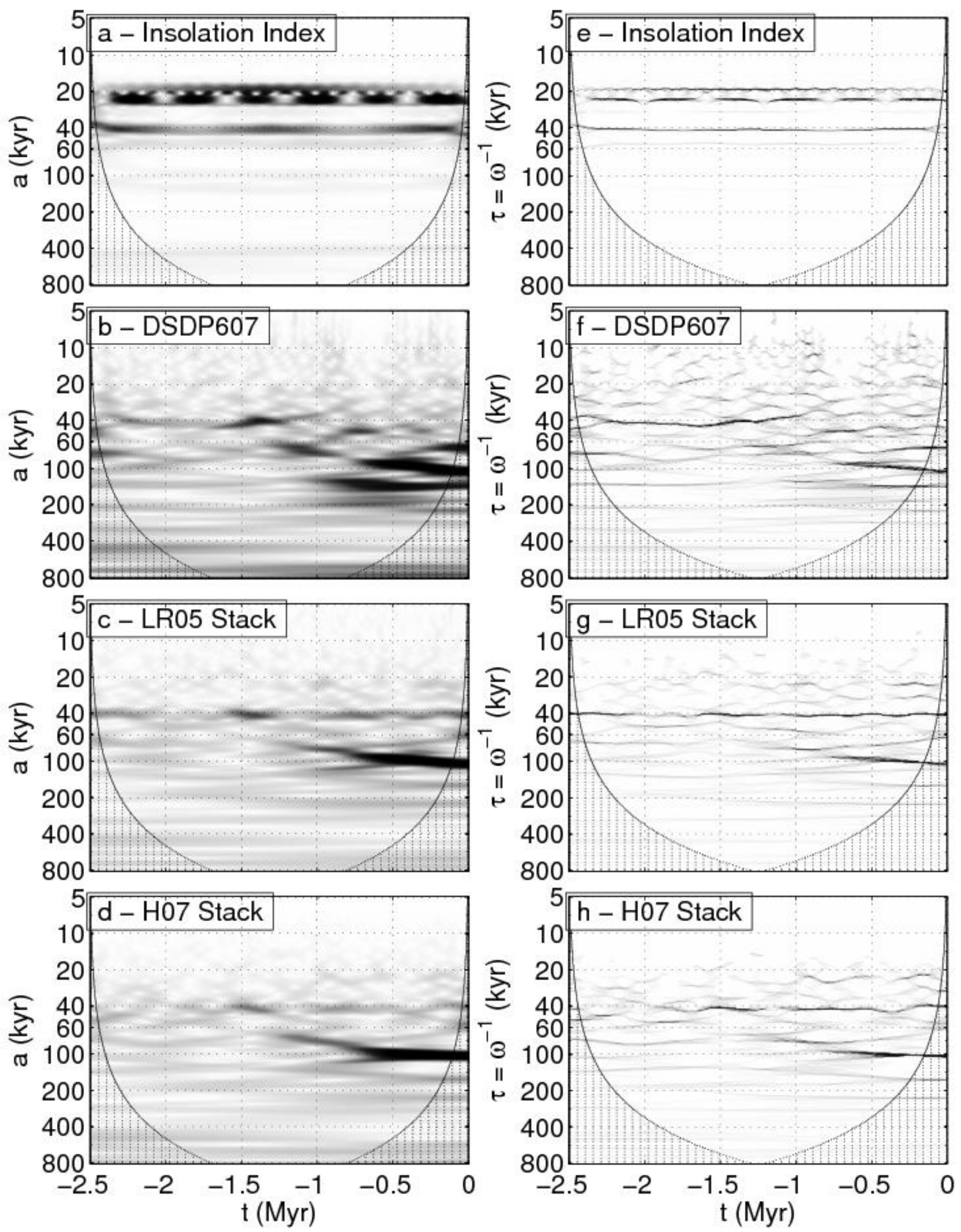}
   \caption{\label{fig:paleo_synsq} The CWT time-scale
     decomposition of (a) the solar forcing index $f_{SF}$, and the climate response in benthic
     ${\delta}^{18}O$ of (b) the DSDP607 core $f_{CR1}$, (c) the LR05
     stack $f_{CR2}$, and (d) the H07 stack $f_{CR3}$. The Synchrosqueezing time-periodicity
     decomposition of (e) the solar forcing index $f_{SF}$, and the climate response in benthic
     ${\delta}^{18}O$ of (f) the DSDP607 core $f_{CR1}$, (g) the LR05 stack $f_{CR2}$, and
	 (h) the H07 stack $f_{CR3}$.}
\end{figure}

The Synchrosqueezing decomposition in the right panels in
Fig. \ref{fig:paleo_synsq} is a far more precise time-frequency
representation of signals from DSDP607 and the stacks than the CWT decomposition in
left panels in Fig. \ref{fig:paleo_synsq} or an STFT analysis of the H07 stack
\cite[Fig. 4]{Huybers2007}. Noise due to local characteristics and measurement errors of
each core is reduced when we shift the spectral analysis from DSDP607 to the stacks, and this is particularly visible in the finer
scales and higher frequencies. In addition, the stacks in Fig. \ref{fig:paleo_synsq}(g) and Fig. \ref{fig:paleo_synsq}(h) show far less
stochasticity above the obliquity band (higher frequencies) compared to DSDP607 in Fig. \ref{fig:paleo_synsq}(f). This enables the 23\,kyr precession cycle
to appear mostly coherent over the last 1\,Myr, especially in comparison to the CWT decompositions. Thanks to the stability of Synchrosqueezing,
the spectral differences below the obliquity band (lower frequencies) are less pronounced betweeen the stacks and DSDP607.
Overall, the stacks show less noisy time-periodicity evolution than DSDP607 or any other single core due to the averaging of multiple, noisy time
series with shared signals. The Synchrosqueezing decompositions are much sharper than the corresponding CWT decompositions, and a time average of the
Synchrosqueezing magnitudes delineates the harmonic components more clearly than a comparable Fourier spectrum (not shown).\\

During the last 2.5\,Myr, the Earth experienced a gradual decrease in
the global long-term temperature and $\text{CO}_2$ concentration, and an
increase in mean global ice volume accompanied with
glacial-interglacial oscillations that have intensified towards the present
(shown in Fig. \ref{fig:paleo_ts}(b) through
Fig. \ref{fig:paleo_ts}(d)). The mid-Pleistocene transition, occurring
abruptly or gradually sometime between 1.2\,Myr and 0.6\,Myr ago, was marked by the
the shift from 41\,kyr-dominated glacial cycles to 100\,kyr-dominated
glacial cycles recorded in deep-sea proxies (e.g., \cite{Ruddiman1986, Mudelsee1997,
Clark2006}). The cause of the emergence of strong 100\,kyr cycle
in the late-Pleistocene climate and incoherency of the precession
band prior to about 1\,Myr (evident in Fig. \ref{fig:paleo_synsq}(g) and
Fig. \ref{fig:paleo_synsq}(h)) are still unresolved questions. Both types of spectral analyses of selected
${\delta}^{18}O$ records indicate that the climate system does not respond linearly to external solar forcing.\\

The Synchrosqueezing decomposition precisely reveals key modulated signals that rise above the stochastic
background. The gain (the ratio of the climate response amplitude to insolation forcing
amplitude) at a given frequency or period, is not constant due to
the nonlinearity of the climate system. The 41\,kyr
obliquity cycle of the global climate response is present almost
throughout the entire Pleistocene in Fig. \ref{fig:paleo_synsq}(g) and
Fig. \ref{fig:paleo_synsq}(h). The most prominent feature of the mid-Plesitocene transition
is the initiation of a lower frequency signal ($\approx$\,70\,kyr)
about 1.2\,Myr ago that gradually evolves into the dominant 100\,kyr
component in the late Pleistocene (starting about 0.6\,Myr
ago). Finding the exact cause for this transition in the dominant ice age
periodicity is beyond the scope of our paper, but the Synchrosqueezing
analysis of the stacks shows that it is not a direct cause-and-effect response
to eccentricity variability (very minor variation of the total
insolation).\\

The precision of the Synchrosqueezing decomposition allows us
to achieve a more accurate inversion across any limited frequency
band of interest than the CWT spectrum. Inverting the Synchrosqueezing
transform over the key orbital periodicity bands (i.e. filtering) in
Fig. \ref{fig:paleo_filtered} emphasizes the nonlinear relationship between
the TOA insolation and climate evolution. The top panels in Fig. \ref{fig:paleo_filtered}, left to
right, show rapidly diminishing contributions to the insolation from precession to eccentricity. However, all of the panels
below the top row in Fig. \ref{fig:paleo_filtered} show a moderately increasing amplitude of variability, i.e., the inverse
cascade of climate response with respect to $f_{SF}$ from the precession to the eccentricity
band in the late-Pleistocene (after $\approx$\,0.6\,Myr). On average, the obliquity band contains more power than the precession
band in DSDP607 and both stacks. Internal feedback mechanisms, most likely due to the long-term cooling of the global climate, amplify the
response of the eccentricity band after the early-Pleistocence (after $\approx$\,1.2\,Myr).   
The cross-band differences in Fig. \ref{fig:paleo_filtered} (rows 2-4) indicate that a superposition of precession
cycles can modulate the climate response in lower frequency bands,
particularly in the eccentricity band, as the climate drifts into a
progressively colder and potentially more nonlinear state (e.g., \cite{Berger1977, Ridgwell1999}). \\

The Synchrosqueezing analysis of the solar insolation index and
benthic $\delta^{18}O$ records makes a new contribution in three
important ways.  First, it produces sharper spectral traces
of a complex system's evolution through the high-dimensional climate state
space than the CWT (e.g., \cite[Fig. 2]{Bolton1995}) or STFT (e.g., \cite[Fig. 2]{Clark2006}).
Second, it better delineates the effects of noise on specific
frequency ranges when comparing a single core to a stack. 
Low frequency components are mostly robust to noise induced by local climate variability,
deposition processes and measurement techniques. Third, Synchrosqueezing allows for a more
accurate reconstruction of the signal components within frequency bands of interest than
the CWT or STFT. Questions about the key processes governing large-scale
climate variability over the last 2.5\,Myr can be answered by using high-precision
data analysis methods such as Synchrosqueezing, in combination with a hierarchy of
dynamical models at various levels of complexity that reproduce the key aspects
of the Pliocene-Pleistocene history. The resulting understanding of past climate
epochs may benefit predictions of the future climate. \cite{Skinner2008}

\begin{figure}[h!]
   \centering
   \includegraphics[width=1\columnwidth]{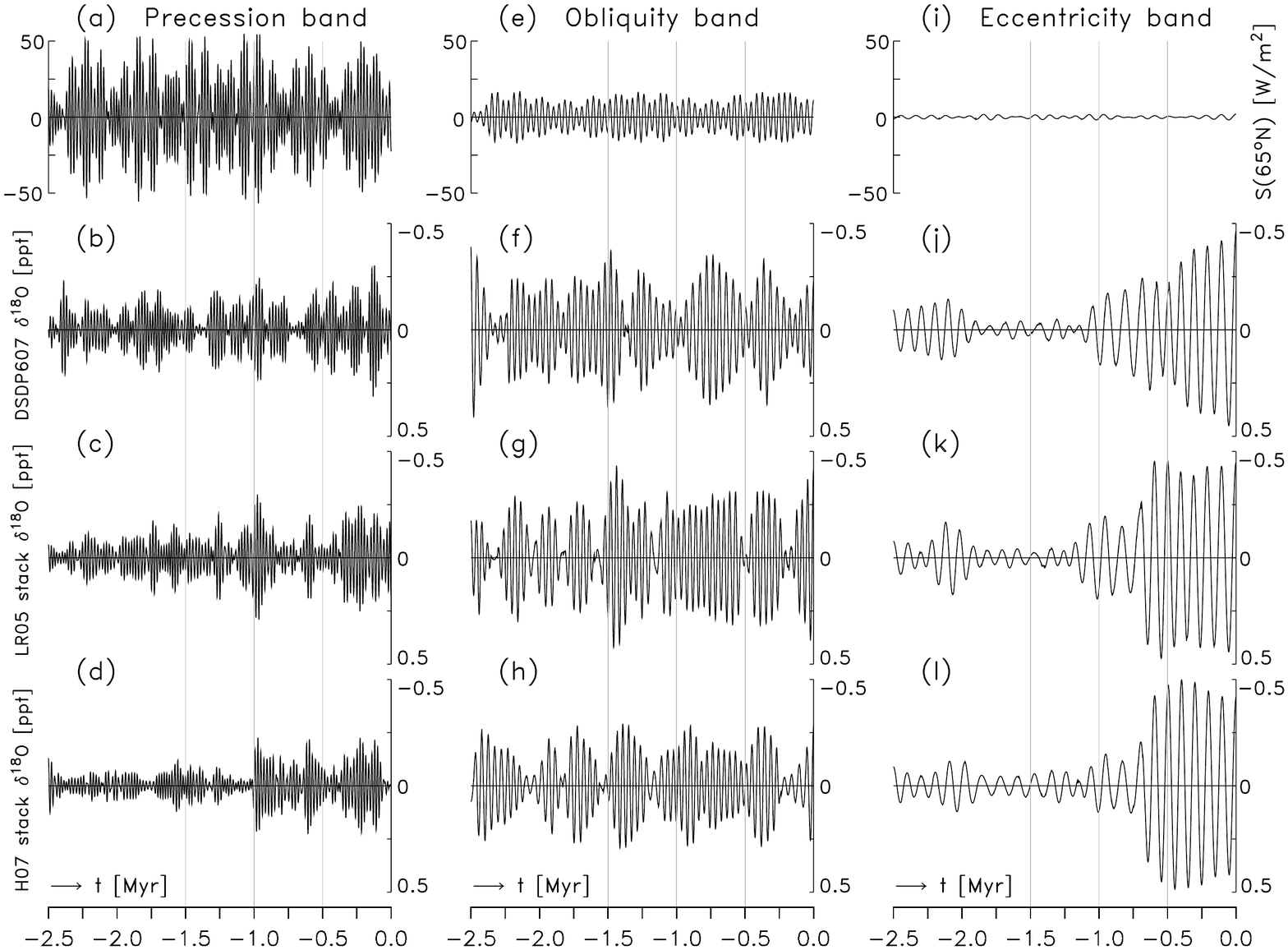}
   \caption{\label{fig:paleo_filtered} Milankovi\'{c} orbital
     components extracted by inverting the Synchrosqueezing transforms
     of the insolation index $f_{SF}$ (a, e, and i), and the climate response in benthic ${\delta}^{18}O$
     from the DSDP607 core $f_{CR1}$ (b, f, and j), the LR05 stack $f_{CR2}$ (c,
     g, and k) and the H07 stack $f_{CR3}$ (d, h, and l). The transforms are
	 inverted over the precession band from 17\,kyr to 25\,kyr (left column), the
	 obliquity band from 40\,kyr to 55\,kyr (middle column), and the eccentricity band
	 from 90\,kyr to 130\,kyr (right column).  }
\end{figure}

\section{Conclusions and Future Directions}
Synchrosqueezing can be used to spectrally analyze and decompose a wide variety of signals
with high precision in time and frequency. An efficient implementation runs in $O(n_v n \log^2 n)$
time and is stable against errors in the signals, both in theory and in practice.  We have
shown how it can be used to gain further insight into the climate evolution of the past 2.5
million years.\\

The authors are also using the Synchrosqueezing transform to study additional topics in climate dynamics,
meteorology and oceanography (climate variability and change, and large-scale teleconnection), as well as
topics in ECG analysis (respiration and T-end detection, \cite{SSTECG}). Synchrosqueezing is also
being used by others to address problems in the analysis of mechanical transmissions
\cite{LL12} and the design of automated trading systems \cite{ATM12}.

\section*{Acknowledgement} 
The authors would like to thank Prof. I. Daubechies for many insightful discussions. N.S. Fu\v{c}kar
also acknowledges valuable input from Dr. O. Elison Timm and H.-T. Wu acknowledges input from Prof. Z. Wu on
the use of EEMD. E. Brevdo acknowledges partial support from NSF GRF grant No. DGE-0646086. H.-T. Wu
and G. Thakur acknowledge partial support from FHWA grant No. DTFH61-08-C-00028. N.S. Fu\v{c}kar
acknowledges support by JAMSTEC, NASA grant No. NNX07AG53G and NOAA grant No. NA11NMF4320128.

\bibliographystyle{plain}
\bibliography{thesis,paleo}

\end{document}